\documentclass[a4paper, 12pt, twoside]{article}
\usepackage{latexsym,bm}
\usepackage[tbtags]{amsmath}
\usepackage{amssymb}
\usepackage{color}
\usepackage{amsmath,amsfonts,amsthm,amssymb,mathrsfs}
\usepackage{float}
\usepackage{indentfirst}
\usepackage{paralist}
\usepackage[english]{babel}
\usepackage{amsmath,amssymb}
\usepackage[linkcolor=blue,citecolor=blue]{hyperref}
\usepackage{cite}
\usepackage{tabularx}
\usepackage{booktabs}
\usepackage{dcolumn}

\usepackage{multirow}

\usepackage{graphicx}

\usepackage{wrapfig}
\usepackage{multicol}
\usepackage{verbatim}

\newcommand{\<}{\langle}
\renewcommand{\>}{\rangle}

\newtheoremstyle{theorem}
{10pt}
{10pt}
{\sl}
{}
{\bf}
{. }
{ }
{}
\newcommand{\ltn}{\ensuremath{\left| \! \left| \! \left|}}
\newcommand{\rtn}{\ensuremath{\right| \! \right| \! \right|}}

\theoremstyle{theorem}

\newtheorem{thm}{Theorem}[section]
\newtheorem{cor}{Corollary}[section]

\newtheorem{de}{Definition}[section]
\newtheorem{lem}{Lemma}[section]
\newtheorem{proposition}{Proposition}[section]
\newtheorem{rem}{Remark}[section]

\numberwithin{equation}{section}

\newtheoremstyle{defi}
{10pt}
{10pt}
{\rm}
{}
{\bf}
{. }
{ }
{}
\theoremstyle{defi}

\allowdisplaybreaks \textheight 21 true cm
\usepackage{epsfig,graphicx,epsf,wrapfig}
\usepackage{dcolumn}
\usepackage{bm}
\usepackage{bbm}
\usepackage[bbgreekl]{mathbbol} 
\begin{document} 


\title{{\large  \textbf{Synchronization of stochastic dissipative differential equation  driven by fractional Brownian motions\thanks{This work
				is supported in part by a NSFC Grant Nos. 12571164 and 12171084, the Fundamental Research Funds for the Central Universities No. RF1028623037, the Postdoctoral Fellowship
Program of CPSF under Grant Number GZC20230438 and  the
China Postdoctoral Science Foundation under Grant Number  2024M760422.}   }}
\author{Qiyong Cao$^{1}$
\footnote{Email:xjlyysx@163.com.} \  \  Hongjun Gao$^{2}$ \footnote{Correspondence, Email: hjgao@seu.edu.cn}\ \ Wei Wei$^{3}$
\\
\footnotesize{  1. School of Mathematics,
		Southwest Jiaotong University, Chengdu 610097, P. R. China}\\ 
\footnotesize{ 2. School of Mathematics,
	Southeast University, Nanjing 211189, P. R. China}\\ 
\footnotesize{3. School of Mathematical Sciences, Nanjing Normal University, Nanjing 210023, P. R. China}
}
 }  

 \date{}
\maketitle 
\footnotesize
\noindent \textbf{Abstract~~} In this paper, we study a class of dissipative stochastic differential equations driven by nonlinear multiplicative fractional Brownian noise with Hurst index $H \in \left(\frac{1}{3},\frac{1}{2})\cup(\frac{1}{2}, 1\right) $. We establish the well-posedness of the associated coupled stochastic differential equations and prove synchronization in the sense of trajectories. Our approach relies on the Doss-Sussmann transformation, which enables us to extend existing results for additive and linear noise to the case of nonlinear multiplicative fractional Brownian noise. The findings provide new insights into the synchronization of dissipative systems under fractional noise perturbations.
\\[2mm]
\textbf{Key words~~~} Stochastic differential equations; synchronization; Doss-Sussmann transformation; fractional Brownian motion.
\\
\textbf{2020 Mathematics Subject Classification~~~} 37L55; 60G22; 34D06

\section{Introduction}
Fractional Brownian motion (fBm), indexed by its Hurst parameter  $H \in \left(\frac{1}{3}, 1\right)$, has become a cornerstone in the modeling of stochastic systems with memory and long-range dependence. Unlike classical Brownian motion, fBm exhibits self-similarity and non-Markovian properties, making it particularly relevant for applications in fields such as climate dynamics, financial modeling, and complex biological processes et al. Extensive research has been devoted to stochastic differential equations (SDEs) driven by fBm or geometric fractional Brownian rough path, with significant contributions focusing on existence and uniqueness of solution (e.g., \cite{MR2387368}, \cite{MR3871629,MR4493559,MR4385780,MR4174393,MR2604669},\cite{MR2378138}),  the structure of invariant measures and random attractors (e.g., \cite{MR4385780},\cite{MR3226746,MR3047956,MR2123208,MR2814425,MR3112925},\cite{MR2946314},\cite{MR4608383}). 

While much is known about general dynamical properties of fBm-driven systems, synchronization in systems driven by fractional noises has received significantly less attention compared to systems driven by standard Wiener or non-Gaussian processes. Synchronization, where the trajectories of previously uncoupled systems asymptotically approach as $\kappa\rightarrow\infty$ after coupling($\kappa$ is the strength of coupling), is a fundamental phenomenon with wide-ranging applications.  On the one hand, in the case of Wiener and non-Gaussian processes-driven systems,  synchronization under additive and multiplicative noise has been analyzed using Lyapunov functionals and coupling methods(e.g., \cite{hu2008stochastic},\cite{jafri2016generalized},\cite{MR4760258}). On the other hand, synchronization can be obtained via coupling in level of random dynamics, such as, random stationary solution, random attractors, random invariant manifold(e.g.,\cite{MR3869887,MR4500127,MR4110684},\cite{MR2154449},\cite{MR2399931},\cite{MR4181483}, \cite{MR4026947,MR4283320,MR4549839,MR2678918},\cite{MR2671384}). However, it requires the noise to be additive or linear  multiplicative. For fractional noise, it exhibits long-range dependence and non-Markovian properties etc., making the previous introduced methods less applicable. Nevertheless, some synchronized results for stochastic system are obtained as considering the simple	fractional  noise cases, such as, additive noise (e.g.,\cite{MR3377401}) and linear multiplicative noise (e.g.,\cite{wei2024synchronization,MR4778574})). Compared to the Wiener and non-Gaussian noise cases, it is necessary to develop some techniques to deal with the synchronization of stochastic systems driven by fractional noises, especially for the case of nonlinear multiplicative fractional noises.

Throughout this paper, we study the synchronization of  a class of stochastic dissipative system driven by $\mathbb{R}^d$-valued fBm $B^H_t$ with Hurst index $H\in\left(\frac{1}{2},1\right)$ or geometric fractional Brownian rough path $\mathbf{B}^{H}=(B^H,\mathbb{B}^H)$ with Hurst index $H\in(\frac{1}{3},\frac{1}{2})$.  Specifically, we consider the following stochastic differential equations
\begin{equation}\label{Sec1:eq 1.1}
	\left\{\begin{array}{l}
		dY_t^1=f(Y_t^1)dt+\sigma(Y_t^1)dB_t^H,\\
		dY_t^2=g(Y_t^2)dt+\sigma(Y_t^2)dB_t^H,
	\end{array}\right.
\end{equation}
and 
\begin{equation}\label{Sec1:eq 1.1r}
	\left\{\begin{array}{l}
		dY_t^1=f(Y_t^1)dt+\sigma(Y_t^1)d \mathbf B_t^H,\\
		dY_t^2=g(Y_t^2)dt+\sigma(Y_t^2)d\mathbf B_t^H,
	\end{array}\right.
\end{equation}
where $t\in[a,b]\subset\mathbb{R}$ and these system with  initial data $(Y_a^1,Y_a^2)$. Additionally, $\mathbb{R}^m$-valued function $f,g$ are locally Lipschitz continuous and dissipative (Sec. 3, Assumption \textbf{A1}).  $\sigma\in\mathcal{L}(\mathbb{R}^d,\mathbb{R}^m)$ is a matrix-valued function satisfying certain boundedness conditions (Sec. 3, Assumption \textbf{A2}).  

Our key contributions are as follows:

\textbf{\textit{A framework to synchronization}:} In rough noise settings (Hurst index $H\in(\frac{1}{3},\frac{1}{2})$), Duc used the Doss-Sussmann transformation to get the well-posedness  and random attractors of equation \eqref{Sec1:eq 1.1r} in Ref. \cite{MR4385780}.  To achieve synchronization of equations \eqref{Sec1:eq 1.1} and \eqref{Sec1:eq 1.1r}, it is essential to introduce a coupling form. Motivated by the additive and linear multiplicative noise cases,  we will determine a specific coupling form through the Doss-Sussmann transformation. To this end, we first consider the following pure Young differential equation
\begin{equation}\label{Sec1:pure Young}
	\begin{aligned}
		dy_t=\sigma(y_t)dB_t^H,\quad t\in[a,b]
	\end{aligned}
\end{equation} 
and  pure rough differential equation
\begin{equation}\label{Sec1:pure rough}
	\begin{aligned}
		dy_t=\sigma(y_t)d\mathbf B_t^H,\quad t\in[a,b]
	\end{aligned}
\end{equation} 
with initial data $y_a$.  Denote the solution operators of these equations    by $\varphi(t,B^H,y_a)$ and $\varphi(t,\mathbf B^H,y_a)$, respectively. Under some assumptions,  $\varphi(t,B^H,y_a)$ and $\varphi(t,\mathbf B^H,y_a)$ are Fr\'{e}chet differentiable with respect to initial data $y_a$, and we denote  the Fr\'{e}chet derivative of  $\varphi(t,B^H,\cdot)$ and $\varphi(t,\mathbf B^H,\cdot)$ by $\frac{\partial \varphi(t,B^H,\cdot)}{\partial y}$ and $\frac{\partial \varphi(t,\mathbf B^H,\cdot)}{\partial y}$, respectively. Similarly, we consider the backward pure Young differential equation
\begin{equation}\label{Sec1:bpure Young}
	\begin{aligned}
		dh_t=\sigma(h_t)dB_t^H,\quad t\in[a,b]
	\end{aligned}
\end{equation} 
and the backward pure rough differential equation
\begin{equation}\label{Sec1:bpure Youngr}
	\begin{aligned}
		dh_t=\sigma(h_t)d\mathbf B_t^H,\quad t\in[a,b]
	\end{aligned}
\end{equation} 
with terminal data $h_b$. Its solution operators are denoted  by $\psi(t,B^H,h_b)$ and $\psi(t,\mathbf B^H,h_b)$, and its Fr\'{e}chet derivatives with respect to terminal data $h_b$ are denoted by $\frac{\partial \psi(t,B^H,\cdot)}{\partial h}$ and $\frac{\partial \psi(t,\mathbf B^H,\cdot)}{\partial h}$, respectively. To realize synchronization of  equations \eqref{Sec1:eq 1.1} and \eqref{Sec1:eq 1.1r}, we construct the following random ordinary differential equations
 \begin{equation}\label{Sec1:zeq}
  \left\{
 \begin{aligned}
 	\dot{Z}^1_t&= \frac{\partial \psi}{\partial h}(a,B^H,\varphi(t,B^H,Z^1_t)) f(\varphi(t,B^H,Z^1_t))+\kappa(Z^2_t-Z^1_t),\\
 \dot{Z}^2_t&= \frac{\partial \psi}{\partial h}(a,B^H,\varphi(t,B^H,Z^2_t)) g(\varphi(t,B^H,Z^2_t))+\kappa(Z^1_t-Z^2_t),
 \end{aligned}
 \right.
\end{equation}
 \begin{equation}\label{Sec1:zeqr}
	\left\{
	\begin{aligned}
		\dot{Z}^1_t&= \frac{\partial \psi}{\partial h}(a,\mathbf B^H,\varphi(t,\mathbf B^H,Z^1_t)) f(\varphi(t,\mathbf B^H,Z^1_t))+\kappa(Z^2_t-Z^1_t),\\
		\dot{Z}^2_t&= \frac{\partial \psi}{\partial h}(a,\mathbf B^H,\varphi(t,\mathbf B^H,Z^2_t)) g(\varphi(t,\mathbf B^H,Z^2_t))+\kappa(Z^1_t-Z^2_t),
	\end{aligned}
	\right.
\end{equation}
with initial data $(Y_a^1,Y^2_a)$ and $t\in[a,b]$, where $\kappa$ is the strength of coupling. Using the Doss-Sussmann transformation $Y_t^1=\varphi(t,B^H,Z_t^1), Y_t^2=\varphi(t,B^H,Z_t^2)$ and $Y_t^1=\varphi(t,\mathbf B^H,Z_t^1), Y_t^2=\varphi(t,\mathbf B^H,Z_t^2)$, we derive the corresponding coupled stochastic differential equations
\begin{equation}\label{Sec1:equi-eq}
	\left\{
	\begin{aligned}
		dY^1_t&= \left(f(Y^1_t)+\kappa\frac{\partial \varphi(t,B^H,Z^1_t)}{\partial y}(Z^2_t-Z^1_t)\right)dt+\sigma(Y^1_t)dB_t^H,\\
		dY^2_t&= \left(g(Y^2_t)+\kappa\frac{\partial \varphi(t,B^H,Z^2_t)}{\partial y}(Z^1_t-Z^2_t)\right)dt+\sigma(Y^2_t)dB_t^H,\\
	\end{aligned}
	\right.
\end{equation} 
and 
\begin{equation}\label{Sec1:equi-eqr}
	\left\{
	\begin{aligned}
		dY^1_t&= \left(f(Y^1_t)+\kappa\frac{\partial \varphi(t,\mathbf B^H,Z^1_t)}{\partial y}(Z^2_t-Z^1_t)\right)dt+\sigma(Y^1_t)d\mathbf B_t^H,\\
		dY^2_t&= \left(g(Y^2_t)+\kappa\frac{\partial \varphi(t,\mathbf B^H,Z^2_t)}{\partial y}(Z^1_t-Z^2_t)\right)dt+\sigma(Y^2_t)d\mathbf B_t^H,\\
	\end{aligned}
	\right.
\end{equation} 
with initial data $(Y_a^1,Y_a^2)$. From the equations \eqref{Sec1:equi-eq} and \eqref{Sec1:equi-eqr}, we can derive  synchronized equations
\begin{equation}\label{Sec1:1.11}
	\begin{aligned}
		d\bar{Y}_t=\frac{1}{2}(f(\bar{Y}_t))+g(\bar{Y}_t))dt+\sigma(\bar{Y}_t)dB_t^H
	\end{aligned}
\end{equation} 
and 
\begin{equation}\label{Sec1:1.12}
	\begin{aligned}
		d\bar{Y}_t=\frac{1}{2}(f(\bar{Y}_t))+g(\bar{Y}_t))dt+\sigma(\bar{Y}_t)d\mathbf B_t^H
	\end{aligned}
\end{equation} 
with initial data $\frac{1}{2}(Y_a^1+Y_a^2)$.

\textbf{\textit{Well-posedness  for equations \eqref{Sec1:equi-eq} and  \eqref{Sec1:equi-eqr} }:} In the classical framework of Young differential equations, it is sufficient to require that the diffusion term $\sigma$ belongs to $C_b^2$, the space of twice continuously differentiable functions with bounded derivatives, to solve Young differential equations. However, for equations \eqref{Sec1:equi-eq} and  \eqref{Sec1:equi-eqr} (as in Ref. \cite{MR4385780}), we need to require that  $\sigma\in C_b^3$,  the space of three times continuously differentiable functions with bounded derivatives.  Indeed, as we consider the Fr\'{e}chet differentiability of $\varphi(t,B^H,\cdot)$ and $\psi(t,B^H,\cdot)$ in Young cases,  stronger regularity of $\sigma$ is necessary. In addition, $C_b^3$-regularity assumption of $\sigma$ is a classical  in rough cases.   Under these assumptions, we can  construct the Doss-Sussmann transformation for the Young cases and rough cases to obtain the global solution of \eqref{Sec1:equi-eq} and \eqref{Sec1:equi-eqr}.

Our models assume that   noises in each component has identical form, it can be interpreted as a deterministic system perturbed by a common random environment.  This simplification allows for a more natural application of the Doss-Sussmann transformation. In contrast, for more general forms of noise (such as,  the noises have different Hurst index in each subsystem),  it may be necessary to treat the system as a whole in vector form, and the consideration of convergence between components becomes more complex. So we intend to do with the difficult cases in future work.    Our method is neither Lyapunov functional method  nor random  dynamical method, it is based on the Doss-Sussmann transformation, the solutions of \eqref{Sec1:equi-eq} and \eqref{Sec1:equi-eqr} may not generate a random dynamical system, then we directly use the trajectories of the solution to construct the synchronization. So our method is direct and natural.  

The remainder of this paper is organized as follows. Section 2 introduces mathematical preliminaries, including some necessary function spaces and  integrals. In Section 3,  we consider the  global well-posedness for a class of coupled equations.  In Section 4, the result on synchronization of stochastic system is given. In Section 5,  we provide an illustrative example to demonstrate the theoretical results.

\setcounter{equation}{0}
\section{Preliminaries}\label{sec:setup}

\subsection{$p$--variation space and H\"{o}lder space}
 Throughout this paper,  we work on the Euclidean space,  equipped with the Euclidean norm $\|\cdot\|$. Additionally, since we will consider multiple component systems, we also work on the space $\mathbb{R}^m\times\mathbb{R}^m$ with norm $\|\cdot\|=(\|\cdot\|^2+\|\cdot\|^2)^{\frac{1}{2}}.$ Let $I=[a,b]\subset\mathbb{R}$ be a compact interval. We denote  the space of all continuous paths   $x:I\mapsto \mathbb{R}^d$ by $C(I;\mathbb{R}^d)$, equipped with the sup norm $\|\cdot\|_{I,\infty}$ given by $\|x\|_{I,\infty}=\sup_{t\in I}\|x_t\|$.  Denote $x_{s,t}=x_t-x_s$. For $p\geq 1$, we denote by $C^{p-var}(I;\mathbb{R}^d)\subset C(I;\mathbb{R}^d)$ the space of continuous paths $x: I\mapsto \mathbb{R}^d$ with  finite $p$--variation semi-norm $\ltn x\rtn_{p-var}:=(\sup_{\mathcal{P}(I)}\sum_{[u,v]\in\mathcal{P}(I)}\|x_{u,v}\|^{p})^{\frac{1}{p}}<\infty$, where the supremum is taken over all finite partitions $\mathcal{P}(I)$  of interval $I$.  Thus,  $C^{p-var}(I;\mathbb{R}^d)$ can be equipped with $p$--variation  norm $\|x\|_{p-var,I}=\|x_a\|+\ltn x\rtn_{p-var,I}$ for any $x\in C^{p-var}(I;\mathbb{R}^d)$. Note that $\ltn x\rtn_{p-var,I}^{p}$ is a control. i.e., $\ltn x\rtn_{p-var,[s,t]}^{p}$ is continuous with respect to $s<t\in I$ and satisfies
\begin{equation}
\ltn x\rtn_{p-var,[s,s]}^{p}=0,\quad \ltn x\rtn_{p-var,[s,u]}^{p}+\ltn x\rtn_{p-var,[u,t]}^{p}\leq \ltn x\rtn_{p-var,[s,t]}^{p},\quad \forall s\leq u\leq t.
\end{equation}
\begin{lem}[Lemma 2.1 \cite{MR3871629}]\label{Sec2: equvalence Control}
	Let $p\geq 1$ and $ x\in C^{p-var}(I;\mathbb{R}^d)$. For any partition of $I$  given by  $a=\tau_0<\tau_1<\cdots<\tau_{N}=b$, we have
	$$\sum_{i=0}^{N-1}{\ltn x \rtn^{p}_{p-var,[\tau_i,\tau_{i+1}]}}\leq \ltn x\rtn_{p-var,I}^p\leq N^{p-1}\sum_{i=0}^{N-1}{\ltn x \rtn^{p}_{p-var,[\tau_i,\tau_{i+1}]}}.$$
\end{lem}
For further properties of the $p$--variation norm, we refer to Ref. \cite{MR2604669}.

In addition, let $\alpha\in(0,1)$. We denote by $C^{\alpha}(I;\mathbb{R}^d)$ the space of  $\alpha$-H\"{o}lder continuous paths on $I$ with the norm
$$\|x\|_{\alpha,I}:=\|x_a\|+\ltn x\rtn_{a,I},\quad \text{where}\quad \ltn x\rtn_{\alpha,I}:=\sup_{s,t\in I,s\neq t}\frac{\|x_{s,t}\|}{|t-s|^{\alpha}}<\infty.  $$
Throughout this paper, we need  to use the concept of a greedy  sequence of stopping times, as introduced in Ref. \cite{MR3112937}. For any given $p\in(1,2)$ and $\gamma\in\mathbb{R}^{+}$, the  sequence of greedy times $\{\tau_{i}(\gamma,I,p)\}_{i\in\mathbb{N}}$ with respect to $p$--variation semi-norm is defined by
\begin{equation}\label{Sec2:pstopping times}
	\tau_0:=a,\quad \tau_{i+1}:=\inf\{t>\tau_i:\ltn x\rtn_{p-var,[\tau_i,t]}=\gamma \}\wedge b.
\end{equation}
Denote by $N_{\gamma,I,p}(x):=\sup\{i\in\mathbb{N}:\tau_i\leq b\}$ the number of stopping times.  It follows  from the first inequlity of Lemma \ref{Sec2: equvalence Control} that
\begin{equation}\label{Sec2:pstopping times numbers}
	N_{\gamma,I,p}(x)\leq 1+\gamma^{-p}\ltn x\rtn_{p-var,I}^p.
\end{equation}

Additionally, let $\gamma\in(0,1)$ and $\alpha\in\left(\frac{1}{2},\alpha^\prime\right),\alpha^\prime<1$.  We can also construct another greedy sequence of stopping times $\{\tau_{i}(\gamma,I, \alpha)\}_{i\in\mathbb{N}}$ with respect to $\alpha$-H\"{o}lder semi-norm given by
\begin{equation}\label{Holder stop}
	\tau_0:=a,\quad \tau_{i+1}:=\inf\{t>\tau_i:(t-\tau_i)^{\alpha}+\ltn x\rtn_{\alpha,[\tau_i,t]}=\gamma \}\wedge b.
\end{equation}
Denote by $N_{\gamma,I,\alpha}(x):=\sup\{i\in\mathbb{N}:\tau_i\leq b\}$ the number of stopping times above. Similarly, we have
\begin{equation}
	N_{\gamma,I,\alpha}\leq 1+\gamma^{-\frac{1}{\alpha^\prime-\alpha}}(b-a)(1+\ltn x \rtn_{\alpha^\prime,I})^{\frac{1}{\alpha^\prime-\alpha}}.
\end{equation}
\begin{rem}
Let $\gamma\in(0,1)$.  For each stopping times interval determined by \eqref{Holder stop}, we have 
 $\ltn x\rtn_{\alpha,[\tau_i,\tau_{i+1}]}<\gamma$ and $(\tau_{i+1}-\tau_i)^\alpha<\gamma$, $i=0,\cdots,N_{\gamma,\alpha,I}-1$.  This property will be used in subsequent sections.
\end{rem}
 \begin{de}
	Let $\alpha \in \left(\frac{1}{3}, \frac{1}{2}\right]$. Denote by $\mathscr{C}^{\alpha}(I; \mathbb{R}^d)$ the space of $\alpha$-Hölder continuous rough paths over $\mathbb{R}^d$. Each element is a pair $\mathbf{x} := (x, \mathbbm{x})$, where $x : I \to \mathbb{R}^d$ and $\mathbbm{x} : \Delta_I \to \mathbb{R}^d \otimes \mathbb{R}^d$, satisfying
	\begin{equation}
		\ltn x\rtn_{\alpha, I} := \sup_{s \neq t \in I} \frac{\|x_{s,t}\|}{|t - s|^{\alpha}} < \infty, \quad
		\ltn \mathbbm{x}\rtn_{2\alpha, \Delta_I} := \sup_{(s,t) \in \Delta_I} \frac{\|\mathbbm{x}_{s,t}\|}{|t - s|^{2\alpha}} < \infty,
	\end{equation}
	where $\Delta_I := \{(s,t) \in I^2 : s < t\}$. Moreover, the pair satisfies the {\upshape Chen} relation:
	\[
	\mathbbm{x}_{s,t} = \mathbbm{x}_{s,u} + \mathbbm{x}_{u,t} + x_{s,u} \otimes x_{u,t}, \quad \forall (s,u,t) \in I^3,~ s \leq u \leq t.
	\] 
	A semi-norm on $\mathscr{C}^{\alpha}(I;\mathbb{R}^d)$ is given by
	\[
	\|\mathbf{x}\|_{\alpha,I}=\ltn x\rtn_{\alpha,I}+\ltn \mathbbm{x}\rtn_{2\alpha,\Delta_I}^{\frac{1}{2}}.
	\]
	Furthermore, let $\mathscr{C}_g^{\alpha}(I;\mathbb{R}^d)$ denote the closure of smooth paths in $\mathscr{C}^{\alpha}(I; \mathbb{R}^d)$ under this topology. $\mathbbm{x}$ is often referred to as the Lévy area or the second order process.
\end{de}

\begin{rem}
	For such rough paths, one may also define a variation-type semi-norm. Let $\alpha\in(\frac{1}{3},\frac{1}{2}]$ and $p = \frac{1}{\alpha}, q = \frac{p}{2}$. Then the $p$--variation semi-norm is defined by
	\[
	\ltn \mathbf{x} \rtn_{p\text{-var}, I} = \left( \ltn x \rtn^p_{p\text{-var}, I} + \ltn \mathbbm{x} \rtn^q_{q\text{-var}, \Delta_I} \right)^{\frac{1}{p}},
	\]
	where
	\[
	\ltn \mathbbm{x} \rtn_{q\text{-var}, \Delta_I} := \left( \sup_{\mathcal{P}(I)} \sum_{[u,v] \in \mathcal{P}(I)} \|\mathbbm{x}_{u,v}\|^q \right)^{\frac{1}{q}}.
	\]
	Therefore,  given  $\gamma\in\mathbb{R}^{+}$, the  sequence of greedy times $\{\tau_{i}(\gamma,I,p)\}_{i\in\mathbb{N}}$ with respect to $p$--variation semi-norm is defined by
	\begin{equation}\label{Sec2:pstopping times1}
		\tau_0:=a,\quad \tau_{i+1}:=\inf\{t>\tau_i:\ltn \mathbf{x}\rtn_{p-var,[\tau_i,t]}=\gamma \}\wedge b.
	\end{equation}
	Denote by $N_{\gamma,I,p}(\mathbf{x}):=\sup\{i\in\mathbb{N}:\tau_i\leq b\}$ the number of stopping times, it also satisfies \begin{equation}\label{Sec2:pstopping times numbers1}
		N_{\gamma,I,p}(\mathbf{x})\leq 1+\gamma^{-p}\ltn \mathbf{x}\rtn_{p-var,I}^p.
	\end{equation}
\end{rem}

\subsection{Young integral}
\begin{de}\label{Sec2:Young integral}
Let $y\in C^{p-var}(I;\mathbb{R}^{m\times d})$  and $x\in C^{q-var}(I;\mathbb{R}^{d})$ with $\frac{1}{p}+\frac{1}{q}>1$.  Young integral $\int_{a}^{b}y_rdx_r$ is defined as
$$\int_{a}^{b}y_rdx_r:=\lim_{|\mathcal{P}(I)|\rightarrow 0}\sum_{[u,v]\in\mathcal{P}(I)}y_ux_{u,v},$$
where the limit is taken over all  the finite partitions $\mathcal{P}(I)$ of $I$ and $|\mathcal{P}(I)|=\max_{[u,v]\in \mathcal{P}(I)}|u-v|$(see Ref. \cite{MR1555421}).
 \end{de}
 Young integral satisfies the additive property and the Young-Loeve estimate(see Ref. \cite[Theorem 6.8]{MR2604669})
 \begin{equation}\label{Sec2:p-Youngbounded}
 \left\|\int_{s}^{t}y_rdx_r-y_sx_{s,t}\right\|\leq C\ltn x \rtn_{p-var,[s,t]}\ltn y \rtn_{q-var,[s,t]},\quad [s,t]\subset I,
\end{equation}
where constant $C>1$.
 \begin{rem}\label{Sec2:Young est-H}
 Young integral  can also be understood in the sense of   H\"{o}lder semi-norm. Specifically, let $\alpha=\frac{1}{p}, \beta=\frac{1}{q}$  and $y\in C^{\alpha}(I;\mathbb{R}^{m\times d})$, $x\in C^{\beta}(I;\mathbb{R}^d)$ with $\alpha+\beta>1$.  Then Young integral $\int_{a}^{b}y_rdx_r$ can be defined   as in Definition \ref{Sec2:Young integral}. Moreover, there exists a constant $C>1$ such that
  $$\left\|\int_{s}^{t}y_rdx_r-y_sx_{s,t}\right\|\leq C\ltn y\rtn_{\alpha,[s,t]}\ltn x\rtn_{\beta,[s,t]}(t-s)^{\alpha+\beta},\quad [s,t]\subset I.$$
 \end{rem}
 \subsection{Rough integral}

 \begin{de}
 	Let $\alpha \in \left( \frac{1}{3}, \frac{1}{2} \right]$. A path $y \in C^\alpha(I;\mathbb{R}^m)$ is said to be controlled by $x \in C^\alpha(I; \mathbb{R}^d)$ if there exists a pair of functions $(y', R^y)$ such that
 	\[
 	y' \in C^\alpha\left(I;\mathbb{R}^{m\times d}\right), \quad R^y \in C^{2\alpha}(\Delta_I; \mathbb{R}^m),
 	\]
 	and 
 	\begin{equation} \label{controlRP}
 		y_{s,t} = y^\prime_s\, x_{s,t} + R^y_{s,t}
 	\end{equation}
 hold	for all $a \leq s \leq t \leq b$.
 	The function $y'$ is called the Gubinelli derivative   of $y$, which is uniquely determined when $x$ is truly rough {\upshape (see \cite{MR4174393})}. 
 	
 	We denote by $\mathcal{D}^{2\alpha}_x(I; \mathbb{R}^m)$ the collection of all such controlled rough path pairs $(y, y')$, which forms a Banach space under the norm
 	\begin{align*}
 		\|(y, y^\prime)\|_{x, 2\alpha, I} &:= \|y_{a}\| + \|y^\prime_{a}\| + \ltn  (y, y^\prime) \rtn_{x, 2\alpha, I}, \\
 		\ltn  (y, y^\prime) \rtn_{x, 2\alpha, I} &:= \ltn y^\prime\rtn_{\alpha, I} + \ltn R^y\rtn_{2\alpha, \Delta_I},
 	\end{align*}
 	where 
 	\[
 	\ltn y^\prime\rtn_{\alpha, I} := \sup_{s \neq t \in I} \frac{\|y^\prime_t - y^\prime_s\|}{|t - s|^\alpha}, \quad \ltn R^y\rtn_{2\alpha, \Delta_I} := \sup_{(s, t) \in \Delta_I} \frac{\|R^y_{s,t}\|}{|t - s|^{2\alpha}}.
 	\]
 \end{de}
 
 \begin{thm}[{\cite{MR4174393}}]\label{rough est}
 	Let $\alpha \in \left(\frac{1}{3}, \frac{1}{2}\right]$. For a fixed rough path $\mathbf{x} = (x,\mathbbm{x}) \in \mathscr{C}^\alpha(I; \mathbb{R}^d)$ and any controlled path $(y, y^\prime) \in \mathcal{D}^{2\alpha}_x(I;\mathbb{R}^m)$, the rough integral $\int_s^t y_u \, \mathrm{d}\mathbf{x}_u$ is defined as the limit
 	\[
 	\int_s^t y_u \, d\mathbf{x}_u := \lim_{|\mathcal{P}(s,t)| \to 0} \sum_{[u,v] \in \mathcal{P}(s,t)} \left( y_u\, x_{u,v} + y^\prime_u\, \mathbbm{x}_{u,v} \right),\quad  [s,t]\subset I. 
 	\]
 	 Moreover, there exists a constant $C_\alpha > 1$ such that the following estimate holds:
 	\begin{align}\label{roughEst}
 		\left\| \int_s^t y_u \, d\mathbf{x}_u - y_s x_{s,t} - y^\prime_s \mathbbm{x}_{s,t} \right\|
 		\leq &C_\alpha |t - s|^{3\alpha} \left( \ltn x\rtn_{\alpha, [s,t]}  \ltn R^y \rtn_{2\alpha, \Delta_{[s,t]}} \right. \nonumber\\
 		&\left. + \ltn y^\prime \rtn_{\alpha, [s,t]}  \ltn \mathbb{x} \rtn_{2\alpha, \Delta_{[s,t]}} \right).
 	\end{align}
 \end{thm}
 
 \begin{rem}
 	In particular, the rough integral above may also be interpreted in the variation norm setting. Let $p = \frac{1}{\alpha}$ with $\alpha \in \left(\frac{1}{3}, \frac{1}{2}\right]$, and suppose $\mathbf{x} \in \mathscr{C}^\alpha(I; \mathbb{R}^d)$, $y \in C^{\alpha}(I;\mathbb{R}^{m\times d})$. Then there exists a constant $C_p>1$ such that  rough integral defined in Theorem~\ref{rough est} satisfies
 	\begin{align}
 		\left\| \int_s^t y_u \, \mathrm{d}\mathbf{x}_u - y_s x_{s,t} - y^\prime_s \mathbb{x}_{s,t} \right\|
 		\leq &C_p \left( \ltn x\rtn_{p\text{-var}, [s,t]} \ltn R^y\rtn_{q\text{-var}, \Delta_{[s,t]}} \right. \nonumber\\
 		&\left. + \|y^\prime\|_{p-var, [s,t]}  \|\mathbbm{x}\|_{q-var, \Delta_{[s,t]}} \right).
 	\end{align}
 \end{rem}
\section{Young  and  rough differential equations}

Let $(\Omega,\mathcal{F},\mathbb{P}^H)$ be a   probability space with repesct to the fractional Brownian motion $B^H$.

 Firstly, we consider the fractional Brownian motion $B^H$ with Hurst index $H>\frac{1}{2}$ in equations \eqref{Sec1:eq 1.1}, \eqref{Sec1:pure Young}, \eqref{Sec1:bpure Young}, \eqref{Sec1:zeq} and \eqref{Sec1:equi-eq}.  According to Kolmogorov's criteria (e.g., \cite{MR1472487}), fractional Brownian motion $\{B^H_t\}_{t\in\mathbb{R}}$ admits an $\alpha^{\prime}$-H\"{o}lder continuous version on any finite interval, where $\frac{1}{2}<\alpha<\alpha^\prime<H$. In other words, almost all trajectories $\Omega\ni W_t$ of $ B_t^H$ are $\alpha^\prime$-H\"{o}lder continuous on any finite interval. Thus, we  use $W_t$ to replace $B_t^H$ in subsequent analysis,
it is clear that  $\alpha^\prime$-H\"{o}lder  continuous version  $W\in C^{\alpha}(I;\mathbb{R}^d)$. 

Moreover, we also consider the geometric fractional Brownian rough path with Hurst index $H\in \left(\frac{1}{3}, \frac{1}{2}\right] $ in equations \eqref{Sec1:eq 1.1r}, \eqref{Sec1:pure rough}, \eqref{Sec1:bpure Youngr}, \eqref{Sec1:zeqr} and \eqref{Sec1:equi-eqr}. According to \cite[Lemma 31]{MR4097587}, the fractional Brownian motion $ B^H$ can almost surely be lifted to a geometric rough path $ \mathbf{B}^H=:(B^H, \mathbb{B}^H) \in \mathscr{C}^{\alpha^\prime}_{g}(I; \mathbb{R}^d) $, where $ \frac{1}{3} < \alpha < \alpha^\prime<H $. That is to say that, for almost all $W\in\Omega$, it can be lifted as   $\mathbf{W}=(W,\mathbb{W})\in \mathscr{C}_{g}^{\alpha^\prime}(I;\mathbb{R}^d)$. Thus, we use $\mathbf{W}$ to repalce $\mathbf{B}^H$ in these equations. 

Furthermore, the coefficients $f,g,\sigma$ satisfy the following  assumptions:
\begin{description}
	\item[A1] The functions $f$ and $g$ are locally Lipschitz and dissipative, i.e.,  there exist constants $D_1\geq 0$,$D_2>0$ such that
	 \begin{equation}\label{Sec3:3.2}
	 	\begin{aligned}
	 		\<y,f(y)\>\leq \|y\|(D_1-D_2\|y\|) \quad\text{and}\quad 	\<y,g(y)\>\leq \|y\|(D_1-D_2\|y\|)~~\forall y\in\mathbb{R}^m;
	 	\end{aligned}
	 \end{equation}
 Moreover, $f$ and $g$ exhibt linear growth  in the perpendicular direction, i.e., there exists $C_{f,g}$ such that
  \begin{equation}\label{Sec3:perpendicular direction}
 	\begin{aligned}
 		&\left\|f(y)-\frac{\<y,f(y)\>}{\|y\|^2}y\right\|\leq C_{f,g}(\|y\|+1)\quad\forall y\in\mathbb{R}^m ~\text{and}~y\neq 0,\\	&\left\|g(y)-\frac{\<y,g(y)\>}{\|y\|^2}y\right\|\leq C_{f,g}(\|y\|+1)\quad\forall y\in\mathbb{R}^m~\text{and}~y\neq 0.
 	\end{aligned}
 \end{equation}
\item [A2] $\sigma\in C_b^3(\mathbb{R}^m;\mathbb{R}^{m\times d})$, i.e., there exists	$C_{\sigma}<\infty$ such that 
\begin{equation}
	\begin{aligned}
	&C_{\sigma}:=\max\{\|\sigma\|_{\infty},\|D\sigma\|_{\infty},\|D^2\sigma\|_{\infty},\|D^3\sigma\|_{\infty}\}.
	\end{aligned}	
\end{equation}
\end{description}
\begin{rem}
	For the above functions $f$ and $g$, define $\pi_yf(y):=\frac{\<y,f(y)\>}{\|y\|^2}y$ and $\pi_yg(y):=\frac{\<y,g(y)\>}{\|y\|^2}y$ for any $y\neq 0$.  Note that \eqref{Sec3:perpendicular direction} is equivalent to the following: for any $y\in\mathbb{R}^m$ and $y\neq 0$, $f$ and $g$ can be decomposed as
		  \begin{equation}\label{Sec3:perpendicular direction-R}
		 	\begin{aligned}
		 		&f(y)=\frac{\<y,f(y)\>}{\|y\|^2}y+\pi_{y}^{\bot}f(y),\quad \text{where}\quad\pi_y^{\bot}=1-\pi_y,~\|\pi_{y}^{\bot}f(y)\|\leq C_{f,g}(\|y\|+1), \\	&g(y)=\frac{\<y,g(y)\>}{\|y\|^2}y+\pi_{y}^{\bot}g(y),\quad \text{where}\quad\pi_y^{\bot}=1-\pi_y,~\|\pi_{y}^{\bot}g(y)\|\leq C_{f,g}(\|y\|+1).
		 	\end{aligned}
		 \end{equation}
\end{rem}

To achieve  the synchronization of the system \eqref{Sec1:eq 1.1} and \eqref{Sec1:eq 1.1r}, we need to  introduce a controller  to couple variables $Y^1$ and $Y^2$.  In general,  let $T_1(\kappa,Y^1,Y^2)$ and $T_2(\kappa,Y^1,Y^2)$ be general controllers,  which  will  be specifically given in the  corresponding cases. Then, for almost all $W\in\Omega$, we consider  the following equations
\begin{equation}\label{Sec3:coupled eq}
	\begin{aligned}
		\left\{\begin{array}{l}
		dY^1_t=f(Y^1_t)dt+T_1(\kappa,Y^1_t,Y^2_t)dt+\sigma(Y^1_t)dW_t,\\
		dY^2_t=g(Y^2_t)dt+T_2(\kappa,Y_t^1,Y^2_t)dt+\sigma(Y^2_t)dW_t,
	\end{array}\right.
	\end{aligned}
\end{equation}
and 
\begin{equation}\label{Sec3:coupled eq1}
	\begin{aligned}
		\left\{\begin{array}{l}
			dY^1_t=f(Y^1_t)dt+T_1(\kappa,Y^1_t,Y^2_t)dt+\sigma(Y^1_t)d\mathbf{W}_t,\\
			dY^2_t=g(Y^2_t)dt+T_2(\kappa,Y_t^1,Y^2_t)dt+\sigma(Y^2_t)d\mathbf{W}_t,
		\end{array}\right.
	\end{aligned}
\end{equation}
where $\kappa\geq 0$ represents the intensity of coupling. We only give the well-posedness of the specific version of \eqref{Sec3:coupled eq} and \eqref{Sec3:coupled eq1}, and  the  well-posedness of uncoupled equation \eqref{Sec1:eq 1.1} and \eqref{Sec1:eq 1.1r} can be obtained in a similar way. Our strategy to get the well-posedness of the specific version of \eqref{Sec3:coupled eq} and   \eqref{Sec3:coupled eq1} is to use the Doss-Sussmann transformation, this technique is introduced by Sussmann \cite{MR461664}.  
Next, our   goal  is to establish the well-posedness of specific  equations  \eqref{Sec1:equi-eq} and \eqref{Sec1:equi-eqr}. 
\begin{thm}\label{Sec:thm 3.1}
	Let $\frac{1}{2}<\alpha<\alpha^\prime<H$.  Assume that  assumptions \textbf{A1} and \textbf{A2} hold.  For almost all $W\in\Omega$, the coupled system \eqref{Sec1:equi-eq} admits a unique solution in $C^{\alpha}([a,b];\mathbb{R}^m\times\mathbb{R}^m)$.
\end{thm}

\begin{thm}\label{Sec:thm 3.1r}
	Let $\frac{1}{3}<\alpha<\alpha^\prime<H\leq \frac{1}{2}$.  Assume that  assumptions \textbf{A1} and \textbf{A2} hold. For almost all $W\in\Omega$,  the coupled system \eqref{Sec1:equi-eqr} admits a unique solution in $\mathcal{D}_{W}^{2\alpha}([a,b];\mathbb{R}^m\times\mathbb{R}^m)$.
\end{thm}
Even though the functions $f,g$ satisfy the assumption \textbf{A1}, it is difficult to verify that  the vector $(f,g)$ is dissipative and linear growth in the perpendicular.   Therefore, we can not directly use the results in  \cite{MR4385780} for rough case. Nonetheless, the proofs  of these two theorem  are similar as in Refs. \cite{MR4385780,MR4493559}. For readers' convenience,   we only give the poof of Theorem \ref{Sec:thm 3.1} in the Appendix \ref{Appendix A}.  For the rough case,  the global well-posedness of \eqref{Sec1:equi-eqr} can be obtained in a similar way as  the proof of Theorem \ref{Sec:thm 3.1}, then we omit its  proof.
 \section{Synchronization}\label{Sec:synchronization}
In this section, we consider the synchronization in the sense of pathwise.  Thus, the noises in equations \eqref{Sec1:eq 1.1}-\eqref{Sec1:1.12} will be replaced by the almost every given path $W\in\Omega$ or its lifted geometric rough path $\mathbf{W}$.  We now have all the necessary  tools  to complete the synchronization of the uncoupled system \eqref{Sec1:eq 1.1} and  \eqref{Sec1:eq 1.1r}.   First, we replace   $(Z^1,Z^2)$ and $(Y^1,Y^2)$ with $(Z^{1,\kappa},Z^{2,\kappa})$ and $(Y^{1,\kappa},Y^{2,\kappa})$  to emphasize  their dependence on the parameter $\kappa$.  	Let $t\in[a,b]$ and $i=1,2$,  define $ \tilde{\psi}(a,W,Z_t^i)=\frac{\partial \psi}{\partial h}(a, W,\varphi(t,W,Z_t^i))-Id$ and $ \tilde{\psi}(a,\mathbf{W},Z_t^i)=\frac{\partial \psi}{\partial h}(a,W,\varphi(t,\mathbf{W},Z_t^i))-Id$. Consequently, for Young case,  $(Z^{1,\kappa},Z^{2,\kappa})$  satisfy  the following random equation
 \begin{equation}\label{Sec4:coupled-ODE}
 	\left\{
 	\begin{aligned}
 		\dot{Z}^{1,\kappa}_t&= (Id+\tilde{\psi}(a,W,Z^{1,\kappa}_t)) f(\varphi(t,W,Z^{1,\kappa}_t))+\kappa(Z^{2,\kappa}_t-Z^{1,\kappa}_t),\\
 		\dot{Z}^{2,\kappa}_t&= (Id+\tilde{\psi}(a,W,Z^{2,\kappa}_t)) g(\varphi(t,W,Z^{2,\kappa}_t))+\kappa(Z^{1,\kappa}_t-Z^{2,\kappa}_t).
 	\end{aligned}
 	\right.
 \end{equation}
 For rough case,  $(Z^{1,\kappa},Z^{2,\kappa})$  satisfy  the following random equation
 \begin{equation}\label{Sec4:coupled-ODE1}
 	\left\{
 	\begin{aligned}
 		\dot{Z}^{1,\kappa}_t&= (Id+\tilde{\psi}(a,\mathbf W,Z^{1,\kappa}_t)) f(\varphi(t,\mathbf W,Z^{1,\kappa}_t))+\kappa(Z^{2,\kappa}_t-Z^{1,\kappa}_t),\\
 		\dot{Z}^{2,\kappa}_t&= (Id+\tilde{\psi}(a,\mathbf W,Z^{2,\kappa}_t)) g(\varphi(t,\mathbf W,Z^{2,\kappa}_t))+\kappa(Z^{1,\kappa}_t-Z^{2,\kappa}_t).
 	\end{aligned}
 	\right.
 \end{equation}
Note that these two equations with  initial data $(Y_a^1,Y_a^2)$.

 Our next goal is to establish the convergence of $Z^{1,\kappa}$ and $Z^{2,\kappa}$. Specifically, for Young case, we aim to prove that there exists a stochastic process $\bar{Z}_t(W)\in C([0,T];\mathbb{R}^m)$  such that 
 $\lim_{\kappa\rightarrow \infty}Z^{1,\kappa}_t(W)=\lim_{\kappa\rightarrow \infty}Z^{2,\kappa}_t(W)=\bar{Z}_t(W)$  holds for almost all trajectories $W\in\Omega$.  Furthermore,  the stochastic process $\bar{Z}_t$ satisfies the following random ordinary differential equation
 \begin{equation}
 	\begin{aligned}
 	\dot{\bar{Z}}_t=\frac{1}{2}(Id+\tilde{\psi}(a,W,\bar{Z}_t))(f(\varphi(t,W,\bar{Z}_t))+g(\varphi(t,W,\bar{Z}_t))) 
 	\end{aligned}
 \end{equation}
 with initial data $\frac{Y_a^1+Y_a^2}{2}$.  
 
 For rough case,  we need to prove that   there exists a stochastic process $\bar{Z}_t(\mathbf W)\in C([0,T];\mathbb{R}^m)$  such that 
 $\lim_{\kappa\rightarrow \infty}Z^{1,\kappa}_t(\mathbf W)=\lim_{\kappa\rightarrow \infty}Z^{2,\kappa}_t(\mathbf W)=\bar{Z}_t(\mathbf W)$  holds for almost all trajectories $W\in\Omega$.  Furthermore,  the stochastic process $\bar{Z}_t$ satisfies the follow random ordinary differential equation
 \begin{equation}
 	\begin{aligned}
 		\dot{\bar{Z}}_t=\frac{1}{2}(Id+\tilde{\psi}(a,\mathbf W,\bar{Z}_t))(f(\varphi(t,\mathbf W,\bar{Z}_t))+g(\varphi(t,\mathbf W,\bar{Z}_t))) 
 	\end{aligned}
 \end{equation}
 with initial data $\frac{Y_a^1+Y_a^2}{2}$.  
 
To obtain the above convergences,  we require an estimate for the  processes $(Y^{1,\kappa},Y^{2,\kappa})$.
 \begin{lem}\label{Sec: Lemma 4.1}
 	Assume that  assumptions \textbf{A1}--\textbf{A2} hold.  For any  $\lambda\in(0,1)$,  there exists $C_{\lambda}>0$ and $\delta_\lambda>0$ such that the solution of \eqref{Sec1:equi-eq}  on the interval $[a,b]$ has the following estimate
 	\begin{equation}\label{Sec4:uniform estimate}
 		\begin{aligned}
 		\|Y^\kappa_t\|\leq e^{-\delta_\lambda(t-a)}\|Y_a\|+C_\lambda( N_{\frac{\lambda}{32C_\sigma C},[a,b],p}(W)+1),
 		\end{aligned}
 	\end{equation}
 it	holds for  almost every given $W\in\Omega$ and $t\in[a,b]$.
 \end{lem}
 \begin{proof}
For almost every given $W\in\Omega$.  Consider  a  sequence of greedy stopping times  $$\left\{\tau_i\right\}_{i=0}^{N_{\frac{\lambda}{32C_\sigma C},[a,b],p}(W)}=\{\tau_{i}(\frac{\lambda}{32C_\sigma C},[a,b],p)\}_{i=0}^{N_{\frac{\lambda}{32C_\sigma C},[a,b],p}(W)},$$ where $32C_\sigma C\ltn W\rtn_{p-var,[\tau_i,\tau_{i+1}]}\leq \lambda$ and  $i=0,\cdots, N_{\frac{\lambda}{32C_\sigma C}),[a,b],p}(W)-1$. For any $t\in[\tau_i,\tau_{i+1}]$ and $i\in\{0,1,\cdots,N_{\frac{\lambda}{32C_\sigma C}),[a,b],p}(W)-1\}$,   from  \eqref{Sec4:coupled-ODE} and Young's inequality,  we have 
 \begin{equation}
 	\begin{aligned}
 	\frac{1}{2}\frac{d\|Z_t^\kappa\|^2}{dt}&=\frac{1}{2}\frac{d(\|Z^{1,\kappa}_t\|^2+\|Z^{2,\kappa}_t\|^2)}{dt}\\
 	&\leq \<(Id+\tilde{\psi}(a,W,Z^1_t)) f(\varphi(t,W,Z^1_t)),Z_t^1\>\\
 	&~~~~+\<(Id+\tilde{\psi}(a,W,Z^2_t)) g(\varphi(t,W,Z^2_t)),Z^2_t\>.
 	\end{aligned}
 \end{equation}
 For  any $i\in\{0,1,\cdots,N_{\frac{\lambda}{32C_\sigma C},[a,b],p}(W)-1\}$, using the same arguments in Ref. \cite[Theorem 2.1]{MR4385780},  there exist constant $
 \bar{C}_\lambda>0$ and $\delta_{\lambda}>0$ such that 
 \begin{equation}\label{Sec4:ineq 4.5}
    \begin{aligned}
    	\frac{d}{dt}(\|Z^{1,\kappa}_t\|^2+\|Z^{2,\kappa}_t\|^2)\leq 2\bar{C}_\lambda-2\delta_{\lambda}(\|Z^{1,\kappa}_t\|^2+\|Z^{2,\kappa}_t\|^2),~~t\in[\tau_i,\tau_{i+1}].
    \end{aligned}
 \end{equation}
 Multiplying  both sides of  equation \eqref{Sec4:ineq 4.5} by $e^{2\delta_{\lambda}t}$,  we obtain 
 \begin{equation}\label{Sec4:ineq 4.6}
 	\begin{aligned}
 		\frac{d}{dt}\left[(\|Z^{1,\kappa}_t\|^2+\|Z^{2,\kappa}_t\|^2)e^{2\delta_{\lambda}t}\right]\leq 2\bar{C}_{\lambda}e^{2\delta_{\lambda}t},~~t\in[\tau_i,\tau_{i+1}],~ i=0,\cdots,N_{\frac{\lambda}{32C_\sigma C},[a,b],p}(W)-1.
 	\end{aligned}
 \end{equation}
For $ i=0,\cdots,N_{\frac{\lambda}{32C_\sigma C}),[a,b],p}(W)-1,$ the inequality  \eqref{Sec4:ineq 4.6} shows that  
 \begin{equation}
 	\begin{aligned}
 		\|Z^{1,\kappa}_t\|^2+\|Z^{1,\kappa}_t\|^2\leq (	\|Z^{1,\kappa}_{\tau_i}\|^2+\|Z^{1,\kappa}_{\tau_i}\|^2)e^{-2\delta_\lambda (t-\tau_i)}+\frac{\bar{C}_{\lambda}}{\delta_{\lambda}},~~t\in[\tau_i,\tau_{i+1}].
 	\end{aligned}
 \end{equation}
This implies that
\begin{equation}\label{Sec4:4.10}
	\|Z_t^\kappa\|\leq \|Z_{\tau_i}^\kappa\|e^{-\delta_\lambda (t-\tau_i)}+\sqrt{\frac{\bar{C}_\lambda}{\delta_\lambda}},\quad \forall t\in[\tau_i,\tau_{i+1}],~ i=0,\cdots,N_{\frac{\lambda}{32C_\sigma C},[a,b],p}(W)-1.
\end{equation}
 By iterating this inequality, we obtain 
 \begin{equation}\label{Sec4:eq 4.8}
 	\begin{aligned}
 		\|Z^\kappa_{\tau_{i+1}}\|\leq \|Z_{\tau_{i}}^\kappa\|e^{-\delta_\lambda (\tau_{i+1}-\tau_i)}+\sqrt{\frac{\bar{C}_\lambda}{\delta_\lambda}},\quad i=0,\cdots,N_{\frac{\lambda}{32C_\sigma C},[a,b],p}(W)-1.
 	\end{aligned}
 \end{equation}
 By repeatedly using inequality \eqref{Sec4:eq 4.8},  for any $i\in\left\{0,1\cdots,N_{\frac{\lambda}{32C_\sigma C},[a,b],p}(W)-1 \right\}$, we know that 
 \begin{equation}\label{Sec4: eq 4.9}
 	\begin{aligned}
 		\|Z^\kappa_{\tau_{i+1}}\|&\leq \|Z_{\tau_{i}}^\kappa\|e^{-\delta_\lambda (\tau_{i+1}-\tau_i)}+\sqrt{\frac{\bar{C}_\lambda}{\delta_\lambda}}\\
 		&\leq \sqrt{\frac{\bar{C}_\lambda}{\delta_\lambda}}+\sqrt{\frac{\bar{C}_\lambda}{\delta_\lambda}}e^{-\delta_\lambda (\tau_{i+1}-\tau_i)}+\|Z_{\tau_{i-1}}^\kappa\|e^{-\delta_\lambda (\tau_{i+1}-\tau_{i-1})}\\
 		&\leq \sqrt{\frac{\bar{C}_\lambda}{\delta_\lambda}}+\sqrt{\frac{\bar{C}_\lambda}{\delta_\lambda}}e^{-\delta_\lambda (\tau_{i+1}-\tau_i)}+\sqrt{\frac{\bar{C}_\lambda}{\delta_\lambda}}e^{-\delta_\lambda (\tau_{i+1}-\tau_{i-1})}\\
 		&~~~~+\|Z_{\tau_{i-2}}^\kappa\|e^{-\delta_\lambda (\tau_{i+1}-\tau_{i-2})}\\
 		&\leq  (i+1)\sqrt{\frac{\bar{C}_\lambda}{\delta_\lambda}}+\|Z_{\tau_0}^\kappa\|e^{-\delta_{\lambda}(\tau_{i+1}-\tau_0)}\\
 		&=N_{\frac{\lambda}{32C_\sigma C}),[a,b],p}(W)\sqrt{\frac{\bar{C}_\lambda}{\delta_\lambda}}+\|Y_a\|e^{-\delta_{\lambda}(\tau_{i+1}-\tau_0)}.\\
 	\end{aligned}
 \end{equation}
 For each $t\in[\tau_i,\tau_{i+1}]$ and $i\in \left\{0,1,\cdots, N_{\frac{\lambda}{32C_\sigma C},[a,b],p}(W)-1\right\}$, the above inequality  and \eqref{Sec4:4.10} implies  that 
 \begin{equation}\label{Sec4: eq 4.10}
 	\begin{aligned}
 	  \|Z_t^\kappa\|&\leq \|Z_{\tau_i}^\kappa\|e^{-\delta_\lambda (t-\tau_i)}+\sqrt{\frac{\bar{C}_\lambda}{\delta_\lambda}}\\
 	  &\leq \|Y_a\|e^{-\delta_\lambda(t-\tau_0)}+\left(N_{\frac{\lambda}{32C_\sigma C},[a,b],p}(W)+1\right)\sqrt{\frac{\bar{C}_\lambda}{\delta_\lambda}}.
 	\end{aligned}
 \end{equation}
 Since $\|\varphi(t,W,Z_t^\kappa)-Z^\kappa_t\|\leq \frac{\lambda}{2}$ (see \eqref{SecA:A40}) on each  stopping times interval $[\tau_i,\tau_{i+1}]$ and $Y_t^\kappa=\varphi(t,W,Z_t^\kappa)$,  there exists a constant $C_{\lambda}:=\frac{\lambda}{2}+\sqrt{\frac{\bar{C}_\lambda}{\delta_{\lambda}}}$ such that 
 \begin{equation}\label{Sec4:eq 4.11}
 	\begin{aligned}
 		\|Y_t^\kappa\|&\leq \|Y_a\|e^{-\delta_\lambda(t-\tau_0)}+\left(N_{\frac{\lambda}{32C_\sigma C},[a,b],p}(W)+1\right)C_\lambda\\
 		&=\|Y_a\|e^{-\delta_\lambda(t-a)}+\left(N_{\frac{\lambda}{32C_\sigma C},[a,b],p}(W)+1\right)C_\lambda.
 	\end{aligned}
 \end{equation}
This completes the proof.
 \end{proof}
 
  \begin{lem}\label{Sec: Lemma 4.1r}
 	Assume that  assumptions \textbf{A1}--\textbf{A2} hold.  For any  $\lambda\in(0,1)$,  there exists $C_{\lambda}>0$ and $\delta_\lambda>0$ such that the solution of \eqref{Sec1:equi-eqr}  on the interval $[a,b]$ has the following estimate
 	\begin{equation}\label{Sec4:uniform estimater}
 		\begin{aligned}
 			\|Y^\kappa_t\|\leq e^{-\delta_\lambda(t-a)}\|Y_a\|+C_\lambda( N_{\frac{\lambda}{16C_\sigma C},[a,b],p}(\mathbf W)+1),
 		\end{aligned}
 	\end{equation}
 	 it	holds for  almost every given $W\in\Omega$ and $t\in[a,b]$, where $\mathbf{W}$ is a geometric rough path lifted by $W$. 
 \end{lem}
 The above lemma can be obtained as the proof of Lemma \ref{Sec: Lemma 4.1} or  Theorem 2.1 in Ref. \cite{MR4385780}. So we omit its proof.
 
To obtain the synchronization of  stochastic systems \eqref{Sec1:eq 1.1} and \eqref{Sec1:eq 1.1r},  we introduce the following transformation:  for any $\kappa\in\mathbb{R}^+$ and $t\in[a,b]$, let
 \begin{equation}\nonumber
   \begin{aligned}
   \bar{Z}^{\kappa}_t=\frac{1}{2}(Z^{1,\kappa}_t+Z^{2,\kappa}_t),\\
   \tilde{Z}^{\kappa}_t=\frac{1}{2}(Z^{1,\kappa}_t-Z^{2,\kappa}_t).
   \end{aligned}
 \end{equation}
 From  \eqref{Sec4:coupled-ODE} and \eqref{Sec4:coupled-ODE1}, $ \bar{Z}^{\kappa}$ and $ \tilde{Z}^{\kappa}$ satisfy the following random differential equations:
 \begin{equation}\label{Sec4:eq 4.12}
   \begin{aligned}
   \dot{\tilde{Z}}_t^\kappa=\frac{1}{2}\left[(Id+\tilde{\psi}(a,W,Z^{1,\kappa}_t)) f(\varphi(t,W,Z^{1,\kappa}_t))-(Id+\tilde{\psi}(a,W,Z^{2,\kappa}_t)) g(\varphi(t,W,Z^{2,\kappa}_t))\right]-2\kappa\tilde{Z}_t^\kappa
   \end{aligned}
 \end{equation}
 and 
  \begin{equation}\label{Sec4:eq 4.12r}
 	\begin{aligned}
 		\dot{\tilde{Z}}_t^\kappa=\frac{1}{2}\left[(Id+\tilde{\psi}(a,\mathbf W,Z^{1,\kappa}_t)) f(\varphi(t,\mathbf W,Z^{1,\kappa}_t))-(Id+\tilde{\psi}(a,\mathbf W,Z^{2,\kappa}_t)) g(\varphi(t,\mathbf W,Z^{2,\kappa}_t))\right]-2\kappa\tilde{Z}_t^\kappa
 	\end{aligned}
 \end{equation}
 with initial data $\frac{1}{2}(Y_a^1-Y_a^2)$. Our goal is to show that $\tilde{Z}_t^\kappa$ converges to $0$ for all $t\in(a,b]$ and almost every given $W\in\Omega$ as $\kappa\rightarrow \infty$.
 \begin{thm}\label{Sec4:Thm 4.1}
   Assume that assumptions \textbf{A1}--\textbf{A2} hold. For any $t\in(a,b]$,  we have 
   \begin{equation}
      \begin{aligned}
      \lim_{\kappa\rightarrow\infty}\|\tilde{Z}_t^\kappa(W)\|=0
      \end{aligned}
   \end{equation}
   holds for almost all $W\in\Omega$.
 \end{thm}  
 \begin{proof}
Let  $\left\{\tau_{i}\right\}_{i=0}^{N_{\frac{\lambda}{32C_\sigma C},[a,b],p}(W)}$  be a  sequence of stopping times used in Lemma \ref{Sec: Lemma 4.1}. For any $t\in[\tau_{i},\tau_{i+1}]$ and $i\in\{0,1,\cdots, N_{\frac{\lambda}{32C_\sigma C},[a,b],p}(W)-1\}$,  using equation \eqref{Sec4:eq 4.12} and Young's inequality, we have 
\begin{equation}\nonumber
	\begin{aligned}
	\frac{d}{dt}\|\tilde{Z}^\kappa_t\|^2=&
	\left\<( Id+\tilde{\psi}(a,W,Z^{1,\kappa}_t))f(\varphi(t,W,Z^{1,\kappa}_t))-( Id+\tilde{\psi}(a,W,Z^{2,\kappa}_t))g(\varphi(t,W,Z^{2,\kappa}_t)),\tilde{Z}^\kappa_t\right\>\\
	&-4\kappa\|\tilde{Z}^{\kappa}_t\|^2\\
	\leq& \frac{1}{4\kappa}(\|(Id+\tilde{\psi}(a,W,Z_t^{1,\kappa}))f(\varphi(t,W,Z_t^{1,\kappa}))\|^2+\|(Id+\tilde{\psi}(a,W,Z_t^{2,\kappa}))g(\varphi(t,W,Z_t^{2,\kappa}))\|^2)\\
	&-2\kappa\|\tilde{Z}^{\kappa}_t\|^2.
	\end{aligned}
\end{equation}
Since $ \|Id+\tilde{\psi}(a,W,Z_t^{\kappa})\|\leq 1+\frac{\lambda}{2}$, $Y_t^{j,\kappa}=\varphi(t,W,Z_{t}^{j,\kappa}),j=1,2$, and Lemma \ref{Sec: Lemma 4.1}, we obtain 
\begin{equation}
	\begin{aligned}
		\frac{d}{dt}\|\tilde{Z}^\kappa_t\|^2
		\leq& \frac{1}{4\kappa}(\|(Id+\tilde{\psi}(a,W,Z_t^{1,\kappa}))f(\varphi(t,W,Z_t^{1,\kappa}))\|^2+\|(Id+\tilde{\psi}(a,W,Z_t^{2,\kappa}))g(\varphi(t,W,Z_t^{2,\kappa}))\|^2)\\
		&-2\kappa\|\tilde{Z}^{\kappa}_t\|^2\\
		\leq& \frac{1}{4\kappa}(1+\frac{\lambda}{2})^2\left(\|f\|_{\infty,B(0,R(Y_a,[a,b],W))}+\|g\|_{\infty,B(0,R(Y_a,[a,b],W)}\right)^2-2\kappa\|\tilde{Z}_t^\kappa\|^2,
	\end{aligned}
\end{equation}
where $B(0,R(Y_a,[a,b],W))$ is a ball centered at the origin with radius  $R(Y_a,[a,b],W)$. Note that radius $R(Y_a,[a,b],W)$ is determined by the estimate \eqref{Sec4:uniform estimate}. Let $$C(f,g)=\left(\|f\|_{\infty,B(0,R(Y_a,[a,b],W))}+\|g\|_{\infty,B(0,R(Y_a,[a,b],W)}\right),$$
then 
\begin{equation}\label{Sec4:eq 4.15}
	\begin{aligned}
		\frac{d}{dt}\|\tilde{Z}^\kappa_t\|^2\leq \frac{1}{4\kappa}(1+\frac{\lambda}{2})^2C^2(f,g)-2\kappa\|Z_t^{\kappa}\|^2.
    \end{aligned}
 \end{equation}
Multiplying $e^{2\kappa t}$ on the both sides of  equation \eqref{Sec4:eq 4.15},  we obtain  
\begin{equation}\nonumber
	\begin{aligned}
	  	\frac{d}{dt}[\|\tilde{Z}^\kappa_t\|^2e^{2\kappa t}]\leq \frac{1}{4\kappa}(1+\frac{\lambda}{2})^2C^2(f,g)e^{2\kappa t},~~ t\in[\tau_i,\tau_{i+1}],
	\end{aligned}
\end{equation}
it implies that 
\begin{equation}\label{Sec4:eq 4.16}
	\begin{aligned}
	  \|\tilde{Z}_t^\kappa\|\leq e^{-\kappa(t-\tau_i)}\|\tilde{Z}_{\tau_i}^\kappa\|+\sqrt{\frac{1}{4\kappa}}(1+\frac{\lambda}{2})C(f,g),\quad t\in[\tau_i,\tau_{i+1}].
	\end{aligned}
\end{equation}
By \eqref{Sec4:eq 4.16}, we have 
\begin{equation}\label{Sec4:eq 4.17}
	\begin{aligned}
		\|\tilde{Z}_{\tau_{i+1}}^\kappa\|\leq e^{-\kappa(\tau_{i+1}-\tau_i)}\|\tilde{Z}_{\tau_i}^\kappa\|+\sqrt{\frac{1}{4\kappa}}(1+\frac{\lambda}{2})C(f,g).
	\end{aligned}
\end{equation}
Hence, we can get 
\begin{equation}\label{Sec4:eq 4.18}
	\begin{aligned}
	\|\tilde{Z}_{\tau_{i+1}}^\kappa\|&\leq e^{-\kappa(\tau_{i+1}-\tau_i)}\|\tilde{Z}_{\tau_i}^\kappa\|+\sqrt{\frac{1}{4\kappa}}(1+\frac{\lambda}{2})C(f,g)\\
		&\leq e^{-\kappa(\tau_{i+1}-\tau_i)}\sqrt{\frac{1}{4\kappa}}(1+\frac{\lambda}{2})C(f,g)+\sqrt{\frac{1}{4\kappa}}(1+\frac{\lambda}{2})C(f,g)\\
		&~~~~+e^{-\kappa(\tau_{i+1}-\tau_{i-1})}\|\tilde{Z}_{\tau_{i-1}}^\kappa\|\\
		&\leq (i+1)\sqrt{\frac{1}{4\kappa}}(1+\frac{\lambda}{2})C(f,g)+e^{-\frac{\kappa}{2}(\tau_{i+1}-\tau_{0})}\|\tilde{Z}_{\tau_{0}}^\kappa\|\\
		&\leq N_{\frac{\lambda}{32C_\sigma C},[a,b],p}(W)\sqrt{\frac{1}{4\kappa}}(1+\frac{\lambda}{2})C(f,g)+e^{-\kappa(\tau_{i+1}-\tau_{0})}\|\tilde{Z}_{\tau_{0}}^\kappa\|.
	\end{aligned}
\end{equation}
By \eqref{Sec4:eq 4.16}, for any $ t\in[\tau_i,\tau_{i+1}]$  and $i\in\{0,1,\cdots, N_{\frac{\lambda}{32C_\sigma C}),[a,b],p}(W)-1\}$, we obtain 
\begin{equation}\label{Sec4:eq 4.19}
	\begin{aligned}
		\|\tilde{Z}_t^\kappa\|&\leq e^{-\kappa(t-\tau_0)}\|\tilde{Z}_{\tau_0}^\kappa\|+(N_{\frac{\lambda}{32C_\sigma C},[a,b],p}(W)+1)\sqrt{\frac{1}{4\kappa}}(1+\frac{\lambda}{2})C(f,g)\\
		&= \frac{1}{2}e^{-\kappa(t-a)}\|Y_a^{1}-Y_a^{2}\|+(N_{\frac{\lambda}{32C_\sigma C},[a,b],p}(W)+1)\sqrt{\frac{1}{4\kappa}}(1+\frac{\lambda}{2})C(f,g).
	\end{aligned}
\end{equation}
From \eqref{Sec4:eq 4.19},  we know that $\lim_{\kappa\rightarrow\infty}\|\tilde{Z}_t^\kappa\|=0$,~$t\in(a,b]$.
 \end{proof}  
 Similarly, we have 
 \begin{thm}\label{Sec4:Thm 4.1r}
 	Assume that assumptions \textbf{A1}--\textbf{A2} hold. For any $t\in(a,b]$,  we have 
 	\begin{equation}
 		\begin{aligned}
 			\lim_{\kappa\rightarrow\infty}\|\tilde{Z}_t^\kappa(\mathbf W)\|=0
 		\end{aligned}
 	\end{equation}
 	hold for almost all $W\in\Omega$, where $\mathbf{W}$ is a geometric rough path lifted by $W$.
 \end{thm}  
 Next, we show  that $Z_t^{1,\kappa}$ and $Z_t^{2,\kappa}$  converge to the same process in the sense of pathwise, and derive the synchronized system.
\begin{thm}\label{Thm:4.2}
	Assume that assumptions \textbf{A1}--\textbf{A2} hold. The system \eqref{Sec1:eq 1.1} can achieve synchronization via the system \eqref{Sec1:equi-eq}, and for almost all $W\in\Omega$,  the corresponding synchronized system is given by
\begin{equation}\label{Sec4:eq 4.26}
	\begin{aligned}
		\bar{Y}_t=\frac{1}{2}(Y_a^1+Y_a^2)+\frac{1}{2}\int_{a}^tf(\bar{Y}_r)+g(\bar{Y}_r)dr+\int_{a}^t\sigma(\bar{Y}_r)dW_r,~~t\in[a,b].
	\end{aligned}
\end{equation}
\end{thm}
\begin{proof} 
	For almost every given $W\in\Omega$. 
Since $\bar{Z}_t^{\kappa}=\frac{1}{2}\left(Z_t^{1,\kappa}+Z_t^{2,\kappa}\right)$ and equation \eqref{Sec4:coupled-ODE},  $\bar{Z}_t^{\kappa}$ satisfies 
\begin{equation}\label{Sec4:eq 4.21}
	\begin{aligned}
		\dot{\bar{Z}}_t^{\kappa}=\frac{1}{2}\left[(Id+\tilde{\psi}(a,W,Z^{1,\kappa}_t)) f(\varphi(t,W,Z^{1,\kappa}_t))+(Id+\tilde{\psi}(a,W,Z^{2,\kappa}_t)) g(\varphi(t,W,Z^{2,\kappa}_t))\right],
	\end{aligned}
\end{equation}
with initial data $\frac{1}{2}(Y_a^1+Y_a^2)$ . 

We first prove that $\bar{Z}_t^\kappa$ is equicontinuous  in $[a,b]$. For any $\varepsilon>0$, let $\delta(\varepsilon):=\frac{2\varepsilon}{(1+\frac{\lambda}{2})C(f,g)}$, if $0<t-s<\delta(\varepsilon)$, then $\|\bar{Z}_{t}^{\kappa}-\bar{Z}_{s}^{\kappa}\|\leq \varepsilon$. Indeed, if $32C_\sigma C\ltn W\rtn_{p-var;[s,t]}\leq\lambda$, then 
\begin{equation}\nonumber
	\begin{aligned}
		\|\bar{Z}^\kappa_t-\bar{Z}^\kappa_s\|&=\frac{1}{2}\left\|\int_{s}^{t}\left[(Id+\tilde{\psi}(a,W,Z^{1,\kappa}_t)) f(\varphi(r,W,Z^{1,\kappa}_r))+(Id+\tilde{\psi}(a,W,Z^{2,\kappa}_r)) g(\varphi(r,W,Z^{2,\kappa}_r))\right]dr\right\|\\
		&\leq \frac{1}{2}\int_{s}^{t}(1+\frac{\lambda}{2})(\|f(Y_r^1)\|+\|g(Y_r^2)\|)dr\\
		&\leq \frac{1}{2}(1+\frac{\lambda}{2})C(f,g)(t-s)\leq \varepsilon.
	\end{aligned}
\end{equation}
If not, i.e., $32C_\sigma C\ltn W\rtn_{p-var;[s,t]}>\lambda$. Consider a sequence of stopping times $\{\tau_i\}_{i=0}^{\hat{N}}$, where $\tau_0=s,32C_\sigma C\ltn W\rtn_{p-var,[\tau_i,\tau_{i+1}]}=\lambda,i\in\{0,\cdots,\hat{N}\}$, $\tau_{\hat{N}}=t$. Hence,
\begin{equation}\nonumber
	\begin{aligned}
		\|\bar{Z}^\kappa_t-\bar{Z}^\kappa_s\|&=\frac{1}{2}\left\|\int_{s}^{t}\left[(Id+\tilde{\psi}(r,W,Z^{1,\kappa}_r)) f(\varphi(t,W,Z^{1,\kappa}_r))+(Id+\tilde{\psi}(r,W,Z^{2,\kappa}_r)) g(\varphi(r,W,Z^{2,\kappa}_r))\right]dr\right\|\\
		&\leq \frac{1}{2}\sum_{i=0}^{\hat{N}-1}\left\|\int_{\tau_i}^{\tau_{i+1}}\left[(Id+\tilde{\psi}(r,W,Z^{1,\kappa}_r)) f(\varphi(r,W,Z^{1,\kappa}_r))+(Id+\tilde{\psi}(r,W,Z^{2,\kappa}_r)) g(\varphi(r,W,Z^{2,\kappa}_r))\right]dr\right\|\\
		&\leq \frac{1}{2}\int_{s}^{t}(1+\frac{\lambda}{2})(\|f(Y_r^1)\|+\|g(Y_r^2)\|)dr\\
		&\leq \frac{1}{2}(1+\frac{\lambda}{2})C(f,g)(t-s)\leq \varepsilon.
	\end{aligned}
\end{equation}

In addition, for almost every given $W\in\Omega$, $\bar{Z}_t^\kappa$ is uniform bounded on $[a,b]$ with respect to variable $\kappa\geq 0$. In fact, according to \eqref{Sec4: eq 4.10}, we know that 
\begin{equation}\nonumber
 	\begin{aligned}
 	  \|Z_t^\kappa\|\leq \|Y_a\|+\left(N_{\frac{\lambda}{32C_\sigma C},[a,b],p}(W)+1\right)\sqrt{\frac{\bar{C}_\lambda}{\delta_\lambda}}
 	\end{aligned}
 \end{equation}
 holds for each $t\in[a,b]$ and almost every given $W\in\Omega$. In view of $\bar{Z}_t^\kappa=\frac{1}{2}(Z_t^{1,\kappa}+Z_t^{2,\kappa})$, then 
 \begin{equation}\nonumber
 	\begin{aligned}
 	  \|\bar{Z}_t^\kappa\|\leq \|Y_a\|+\left(N_{\frac{\lambda}{32C_\sigma C},[a,b],p}(W)+1\right)\sqrt{\frac{\bar{C}_\lambda}{\delta_\lambda}}.
 	\end{aligned}
 \end{equation}
Then for any given subsequence $\{\bar{Z}^{\kappa_n}\}_{n\in\mathbb{N}}$, using Arzela-Ascoli theorem to induce that there exists a subsequence $\kappa_{n_j}\rightarrow\infty$ such that $\bar{Z}^{\kappa_{n_j}}_t(W)\rightarrow \bar{Z}_t(W), t\in[a,b]$ as $n_j\rightarrow \infty$.   In particular,  for $t\in(a,b]$ and almost every given $W\in\Omega$, we have 
\begin{equation}\label{Sec4:eq 4.22}
	\begin{aligned}
		\bar{Z}^{\kappa_{n_{j}}}_t-Z^{1,\kappa_{n_{j}}}_t=\frac{1}{2}(Z^{2,\kappa_{n_{j}}}_{t}-Z^{1,\kappa_{n_{j}}}_{t})\to0,\\
		\bar{Z}^{\kappa_{n_{j}}}_t-Z^{2,\kappa_{n_{j}}}_t=\frac{1}{2}(Z^{1,\kappa_{n_{j}}}_{t}-Z^{2,\kappa_{n_{j}}}_{t})\to0,
	\end{aligned}
\end{equation}
as $\kappa_{n_j}\rightarrow\infty$. Therefore, the above limits \eqref{Sec4:eq 4.22} means that 
\begin{equation}
	\begin{aligned}
		Z^{1,\kappa_{n_{j}}}_t=2\bar{Z}^{\kappa_{n_{j}}}_t-Z^{2,\kappa_{n_{j}}}_t=\bar{Z}^{\kappa_{n_{j}}}_t+(\bar{Z}^{\kappa_{n_{j}}}_t-Z^{2,\kappa_{n_{j}}}_{t})\to \bar{Z}_t+0= \bar{Z}_t,\\
		Z^{2,\kappa_{n_{j}}}_t=2\bar{Z}^{\kappa_{n_{j}}}_t-Z^{1,\kappa_{n_{j}}}_t=\bar{Z}^{\kappa_{n_{j}}}_t+(\bar{Z}^{\kappa_{n_{j}}}_t-Z^{1,\kappa_{n_{j}}}_{t})\to \bar{Z}_t+0= \bar{Z}_t,
	\end{aligned}
\end{equation}
hold for all $t\in(a,b]$ and almost every given $W\in\Omega$. Since Theorem \ref{Sec4:Thm 4.1} holds for the different sequence $\{\kappa_n\}$ and $t\in(a,b]$, i.e., $(Z_t^{1,\kappa_n},Z_t^{2,\kappa_n})$, and its subsequence $(Z_t^{1,\kappa_{n_j}},Z_t^{2,\kappa_{n_j}})$ has the same limit $(\bar{Z}_t,\bar{Z}_t)$. Therefore,  we  can get $\bar{Z}_t^\kappa\rightarrow \bar{Z}_t$ for all $t\in[a,b]$ and almost every given $W\in\Omega$ as $\kappa\rightarrow\infty$. Furthermore, using integral relation 
\begin{equation}
	\begin{aligned}
		\bar{Z}^{\kappa}_t=&\frac{1}{2}(Y_a^1+Y_a^2)\\
		&+\int_{a}^t\frac{1}{2}\left[(Id+\tilde{\psi}(a,W,Z^{1,\kappa}_r)) f(\varphi(r,W,Z^{1,\kappa}_r))+(Id+\tilde{\psi}(a,W,Z^{2,\kappa}_r)) g(\varphi(r,W,Z^{2,\kappa}_r))\right]dr,
	\end{aligned}
\end{equation}
the Dominated Convergence Theorem, the continuity of $\tilde\psi(a,W,\cdot)$ and $\varphi(t,W,\cdot)$ with respect to the terminal data and the initial data, respectively, we have 
\begin{equation}
	\begin{aligned}
		\bar{Z}_t=&\frac{1}{2}(Y_a^1+Y_a^2)\\
		&+\int_{a}^t\frac{1}{2}\left[(Id+\tilde{\psi}(a,W,\bar Z_r)) f(\varphi(r,W,\bar Z_r))+(Id+\tilde{\psi}(a,W,\bar Z_r)) g(\varphi(r,W,\bar Z_r))\right]dr
	\end{aligned}
\end{equation}
hold for all $t\in[a,b]$ and almost every given $W\in\Omega$. Finally, for almost every given $W\in\Omega$, given the one-to-one correspondence of Doss-Sussmann transformation $\bar{Y}_t=\varphi(t,W,\bar{Z}_t)$, we can get a stochastic differential equation in pathwise sense, i.e.
\begin{equation}\label{Sec4:eq 4.20}
	\begin{aligned}
		\bar{Y}_t=\frac{1}{2}(Y_a^1+Y_a^2)+\frac{1}{2}\int_{a}^tf(\bar{Y}_r)+g(\bar{Y}_r)dr+\int_{a}^t\sigma(\bar{Y}_r)dW_r 
	\end{aligned}
\end{equation}
hold  for almost every given $W\in\Omega$. This means that uncoupled system \eqref{Sec1:eq 1.1}  can achieve synchronization via  a stochastic coupled system  \eqref{Sec1:equi-eq}, which is regarded as the synchronized operation. Finally,  the synchronized system is equation \eqref{Sec4:eq 4.20}.
\end{proof} 
Using the same method, we have 
\begin{thm}\label{Thm:4.2r}
	Assume that assumptions \textbf{A1}--\textbf{A2} hold. The system \eqref{Sec1:eq 1.1r} can achieve synchronization via the system \eqref{Sec1:equi-eqr}, and  for almost all $W\in\Omega$,  the corresponding synchronized system is given by
	\begin{equation}\label{Sec4:eq 4.26r}
		\begin{aligned}
			\bar{Y}_t=\frac{1}{2}(Y_a^1+Y_a^2)+\frac{1}{2}\int_{a}^tf(\bar{Y}_r)+g(\bar{Y}_r)dr+\int_{a}^t\sigma(\bar{Y}_r)d\mathbf W_r,~~t\in[a,b],
		\end{aligned}
	\end{equation}
	where $\mathbf{W}$ is a geometric rough path lifted by $W$.
\end{thm}
\section{Example}
We consider  a system of two particles evolving  in the  double-well  potential under the influence of  stochastic force.  Specifically, the system is described by  the following stochastic differential equation 
\begin{equation}\label{Sec5:5.1}
		\left\{
	\begin{aligned}
		dY^1_t&= \left(Y^1_t-(Y^1_t)^3\right)dt+sin (Y^1_t)dB^H_t,\\
		dY^2_t&=\left(Y^2_t-(Y^2_t)^3\right)dt+sin(Y^2_t)dB^H_t,\\
	\end{aligned}
	\right.
\end{equation}
with initial data $(Y_0^1,Y_0^2)=(1,3)$ and $t\in[0,1]$, where $B^H$ is a fractional Brownian motion with Hurst index $H=0.7$.   

To achieve the synchronization of  system \eqref{Sec5:5.1},  we introduce  the following coupled system 
\begin{equation}\label{Sec5:5.2}
	\left\{
	\begin{aligned}
		dY^1_t&= \left(Y^1_t-(Y^1_t)^3+\kappa\frac{\partial \varphi(t,B^H,Z^1_t)}{\partial y}(Z^2_t-Z^1_t)\right)dt+sin(Y^1_t)dB^H_t,\\
		dY^2_t&= \left(Y^2_t-(Y^2_t)^3+\kappa\frac{\partial \varphi(t,B^H,Z^2_t)}{\partial y}(Z^1_t-Z^2_t)\right)dt+sin(Y^2_t)dB^H_t,\\
	\end{aligned}
	\right.
\end{equation}
with initial data $(Y_0^1,Y_0^2)=(1,3)$,  where  $Z=(Z^1,Z^2)$ satisfies the auxiliary equation
 \begin{equation}\label{Sec5:5.3}
 	\left\{
 	\begin{aligned}
 		\dot{Z}^1_t&= \frac{\partial \psi}{\partial h}(0,B^H,\varphi(t,B^H,Z^1_t)) (\varphi(t,B^H,Z_t^1)-\varphi^3(t,B^H,Z_t^1))+\kappa(Z^2_t-Z^1_t),\\
 		\dot{Z}^2_t&= \frac{\partial \psi}{\partial h}(0,B^H,\varphi(t,B^H,Z^2_t)) (\varphi(t,B^H,Z_t^2)-\varphi^3(t,B^H,Z_t^2))+\kappa(Z^1_t-Z^2_t),
 	\end{aligned}
 	\right.
 \end{equation}
 with initial data $Z_0=(1,3)$ and $t\in[0,1]$. For $t\in[0,1]$.  Let  $\varphi(r,B^H,Z_t^j),j=1,2,r\in[0,t],$ represent the solution maps of the pure Young differential equations  
  \begin{equation}\label{Sec5:5.4}
 	\left\{
 	\begin{aligned}
 	dy_r^j&= sin(y_r^j)dB^H_r, ~~r\in[0,t],\\
 		y_0^j&=Z_t^j,
 	\end{aligned}
 	\right.
 \end{equation}
and    $\frac{\partial \varphi(r,B^H,Z^j_t)}{\partial y},j=1,2, r\in[0,t],$ are the Fr\'{e}chet derivative of  $y_t$ with respect to  initial data, it satisfies the following equations
 \begin{equation}\label{Sec5:5.5}
	\left\{
	\begin{aligned}
		d\xi_r^j&= cos(y_r^j)\xi_r^jdB^H_r,~~r\in[0,t],\\
		\xi_0^j&=1.
	\end{aligned}
	\right.
\end{equation}
Similarly, $\psi(r,B^H,\varphi(t,B^H,Z_t^j)),j=1,2,r\in[0,t]$, are the solution maps of backward Young differential equations
 \begin{equation}\label{Sec5:5.6}
	\left\{
	\begin{aligned}
		dh_r^j&= sin(h_r^j)dB^H_r,~~r\in[0,t],\\
		h_t^j&=\varphi(t,B^H,Z_t^j),
	\end{aligned}
	\right.
\end{equation}
and  $\frac{\partial \psi(r,B^H,\varphi(t,B^H,Z_t^j))}{\partial h},j=1,2,r\in[0,t],$ are the  Fr\'{e}chet derivative of $h_t$ with  terminal data, it satisfies the following equations
 \begin{equation}\label{Sec5:5.7}
 	\left\{
 	\begin{aligned}
 		d\rho_r^j&= cos(h_r^j)\rho_r^jdB^H_r,~~r\in[0,t],\\
 		\rho_t^j&=1.
 	\end{aligned}
 	\right.
 \end{equation}
 We note that the functions $f(x)=g(x)=x-x^3$, it satisfies the assumption \textbf{A1} as discussed in Ref. \cite{MR4385780}. Additionally, $sin(x)$ is also satisfies the assumption \textbf{A2}.  Therefore, by Theorem \ref{Thm:4.2},  the system  \eqref{Sec5:5.1} can  achieve synchronization through the  coupled form  \eqref{Sec5:5.2}  as $\kappa\rightarrow \infty$.  The synchronized equation is given by 
 \begin{equation}\label{Sec5:5.8}
     \begin{aligned}
         dY_t=(Y_t-Y_t^3)dt+sin(Y_t)dB^H_t
     \end{aligned}
 \end{equation}
 
 To validate our theoretical results, we perform numerical simulations using MATLAB. Figure \ref{fig:23} displays the trajectories of  $Y^1_t,Y^2_t$ and $Y_t$   for   $\kappa=0,10,100,100$.
 \begin{figure}[H]
 	\centering
 	\includegraphics[width=\textwidth]{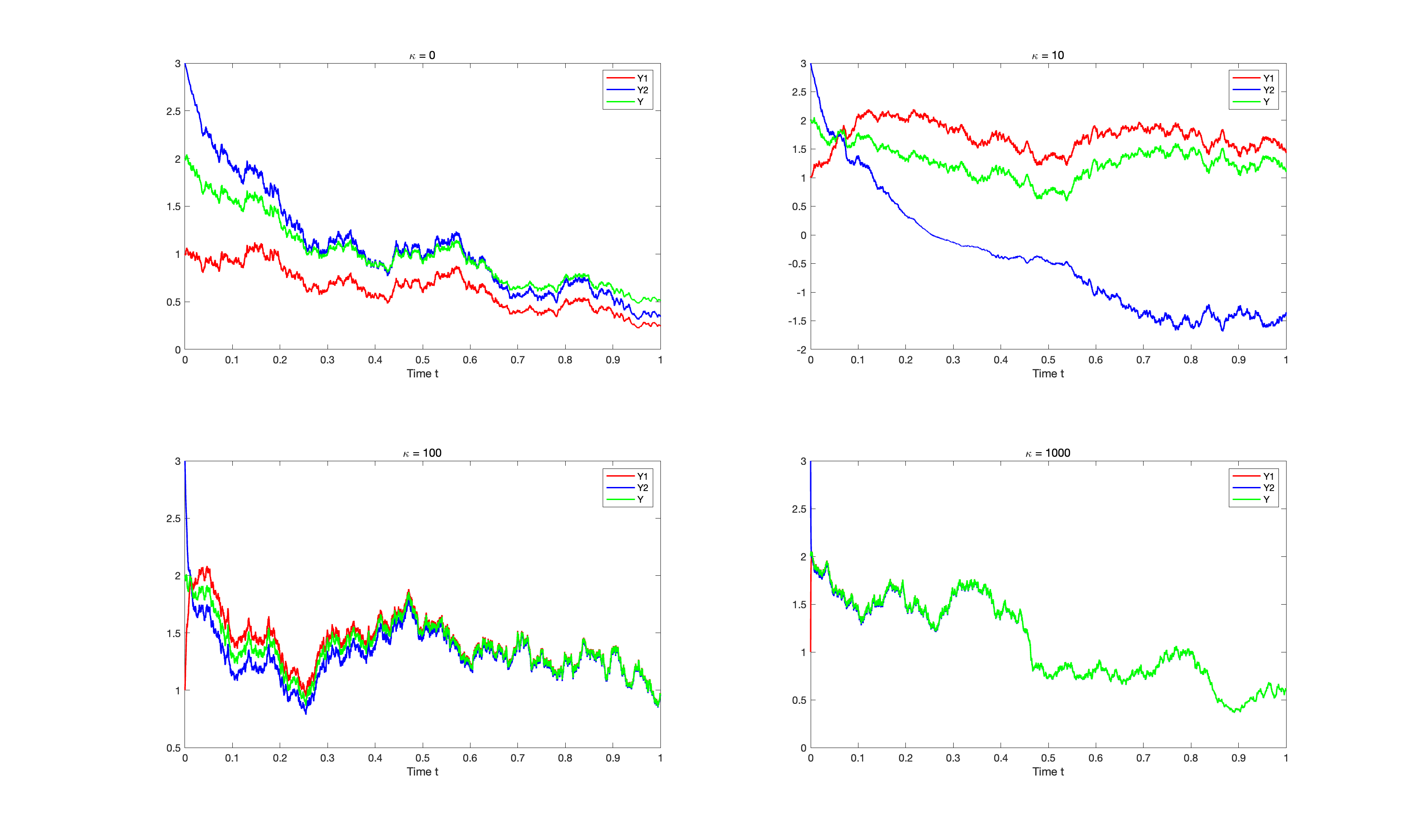}
 	\caption{Trajectories of $Y^1_t$,$Y^2_t$,$Y_t$ for $\kappa=0,10,100,1000$.}
 	\label{fig:23}
 \end{figure}
Figure \ref{fig:23} demonstrates that as the coupling parameter 
$\kappa$ increases, the trajectories of $Y^1_t$ and $Y^2_t$ converge to the trajectory of  $Y_t$ on the interval $(0,1]$. This  phenomenon means that  the system  \eqref{Sec5:5.1}  via the  coupled form  \eqref{Sec5:5.2} to  almost achieve synchronization on $(0,1]$ as $\kappa=100$, it reflects the state of components $Y^1_t$ and $Y^2_t$ can be effectively tracked by the state of $Y$ for sufficiently large $\kappa$. 

For the systems \eqref{Sec5:5.1}-\eqref{Sec5:5.8}, we can also replace fractional Brownian motion $B^H$ by  the geometric fractional Brownian rough path $\mathbf{W}$ with Hurst index $H=0.4$.  The trajectories of  $Y^1_t,Y^2_t$ and $Y_t$   for   $\kappa=0,10,100,100$ as follows:
 \begin{figure}[H]
	\centering
	\includegraphics[width=\textwidth]{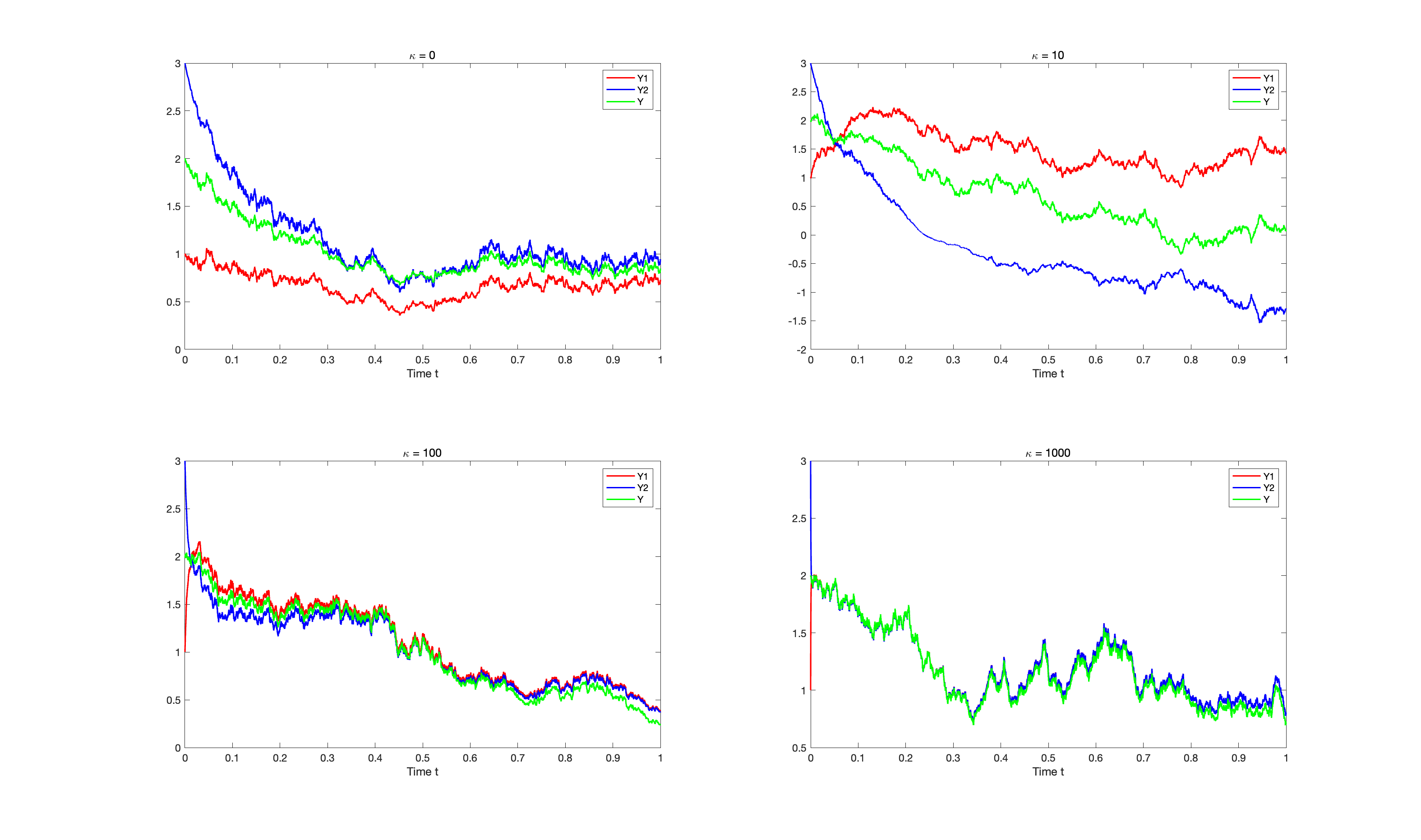}
	\caption{Trajectories of $Y^1_t$,$Y^2_t$,$Y_t$ for $\kappa=0,10,100,1000$.}
	\label{fig:24}
\end{figure}
Therefore, the figure  \eqref{fig:24} also shows the synchronized phenomenon. It reflects that the controller, we designed, is effective to realize synchronization for rough case.

In conclusion, both theoretical analysis and numerical simulations confirm that the system \eqref{Sec5:5.1} achieves synchronization through the coupled form \eqref{Sec5:5.2} as  $\kappa\rightarrow \infty$,  with the synchronized system described by \eqref{Sec5:5.8}, so is rough case.

\appendix
\section{The proof of Theorem \ref{Sec:thm 3.1}}\label{Appendix A}
Before giving the proof of Theorem \ref{Sec:thm 3.1},  we first give the well-posedness of the pure Young differential equation \eqref{Sec1:pure Young}.
\begin{thm}\label{Sec3:Thm3.1}
Let $I = [a, b] \subset \mathbb{R}$. If assumption \textbf{A2} holds, then the pure Young differential equation \eqref{Sec1:pure Young} admits a unique solution in  $C^{\alpha}(I; \mathbb{R}^m)$ for almost every $W \in \Omega$.
\end{thm}
\begin{proof}
We establish the global well-posedness of the solution to \eqref{Sec1:pure Young} by applying the Banach fixed-point theorem in conjunction with a concatenation technique. Specifically, we first prove the well-posedness on each stopping time interval determined by \eqref{Holder stop} via the Banach fixed-point theorem, and then concatenate these local solutions to obtain well-posedness over the entire time interval.
	
For almost every given $W\in\Omega$.	Let $\gamma=\frac{1}{2
		\tilde{C}_{\sigma}}$ in \eqref{Holder stop}, where $\tilde{C}_{\sigma}:=2CC_\sigma \vee 1$, and define the following  space
	$$S:=\left\{y\in C^{\alpha}([a,\tau_1];\mathbb{R}^m): \text{with initial data~}y_a,~\text{and}~\ltn y\rtn_{\alpha,[a,\tau_1]}\leq 1\right\},$$
	it is not difficult to prove that $S$ is a Banach space under the semi-norm $\ltn \cdot\rtn_{\alpha,[a,\tau_{1}]}$. One defines the map $\mathcal{M}$ by
	\begin{equation}
		\mathcal{M}(y)(t):=y_a+\int_{a}^{t}\sigma(y_s)dW_s,\quad t\in [a,\tau_1].
	\end{equation}
	\textbf{\textit{Invariance:}} For any $y\in S$ and $s<r\in[a,\tau_1]$, the estimate  in  Remark \ref{Sec2:Young est-H} yields
	\begin{equation}
		\begin{aligned}
			\| \mathcal{M}(y)_{s,r}\|&=\left\|\int_{s}^{r}\sigma(y_{s^\prime})dW_{s^\prime}\right \|\leq \|\sigma(y_s)W_{s,r}\|+C\ltn\sigma(y)\rtn_{\alpha,[s,r]}\ltn W \rtn_{\alpha,[s,r]}(r-s)^{2\alpha}\\
			&\leq \|\sigma\|_{\infty}\ltn W\rtn_{\alpha,[s,r]}(r-s)^{\alpha}+C\ltn\sigma(y)\rtn_{\alpha,[s,r]}\ltn W \rtn_{\alpha,[s,r]}(r-s)^{2\alpha}.
		\end{aligned}		
	\end{equation}
	Note that $\|\sigma\|_{\infty}=\|\sigma(y)\|_{\infty}\leq C_\sigma$, and
	\begin{equation}
		\begin{aligned}
			\|\sigma(y_t)-\sigma(y_s)\|	&\leq C_\sigma\|y_t-y_s\|, \quad s,t\in I,
		\end{aligned}
	\end{equation}
	induces that
	$$\ltn \sigma(y)\rtn_{\alpha,[s,r]}\leq C_\sigma\ltn y\rtn_{\alpha,[s,r]}.$$
	Hence,
	$$\ltn \mathcal{M}(y)\rtn_{\alpha,[a,\tau_1]}\leq C_\sigma\ltn W\rtn_{\alpha,[a,\tau_1]}+CC_\sigma\ltn y\rtn_{\alpha,[a,\tau_1]}\ltn W\rtn_{\alpha,[a,\tau_1]}(\tau_1-a)^\alpha,$$
	then   we get that
	$$\ltn\mathcal{M}(Y)\rtn_{\alpha,[\tau_0,\tau_1]}\leq 1,$$
	it implies the invariance of the mapping $\mathcal{M}$.
	
	\textbf{\textit{Contraction:}} For any $y\in S$ and $\tilde{y}\in S$,
	\begin{equation}
		\begin{aligned}
			\mathcal{M}(y)(t)-\mathcal{M}(\tilde{y})(t)	=\int_{a}^{t}\sigma(y_{s^\prime})-\sigma(\tilde{y}_{s^\prime})dW_{s^\prime}, \quad t\in[a,\tau_1].
		\end{aligned}
	\end{equation}
	For any $s<r\in[a,\tau_1]$, by Remark \ref{Sec2:Young est-H} we have
	\begin{equation}\label{Sec3:3.14}
		\begin{aligned}
			&\|\mathcal{M}(y)(r)-\mathcal{M}(\tilde{y})(r)-\mathcal{M}(y)(s)+\mathcal{M}(\tilde{y})(s)\|=\left\|\int_{s}^{r}\sigma(y_{s^\prime})-\sigma(\tilde{y}_{s^\prime})dW_{s^\prime}\right\| \\
			&\leq \|\sigma(y_s)-\sigma(\tilde{y}_s)\|\|W_{s,r}\|+C\ltn \sigma(y)-\sigma(\tilde{y})\rtn_{\alpha,[s,r]}\ltn W\rtn_{\alpha,[s,r]}(r-s)^{2\alpha}\\
			&\leq \|\sigma(y_s)-\sigma(\tilde{y}_s)\|\ltn W\rtn_{\alpha,[a,\tau_1]}(r-s)^{\alpha}+C\ltn \sigma(y)-\sigma(\tilde{y})\rtn_{\alpha,[a,\tau_1]}\ltn W\rtn_{\alpha,[a,\tau_1]}(r-s)^{2\alpha}.
		\end{aligned}
	\end{equation}
	Note that for any $s\in[a,\tau_1]$, we have
	\begin{equation}
		\begin{aligned}
			\|\sigma(y_s)-\sigma(\tilde{y}_s)\|\leq C_{\sigma}\|y_s-\tilde{y}_s\|\leq C_\sigma\ltn y-\bar{y}\rtn_{\alpha,[a,\tau_1]}(\tau_1-a)^\alpha,
		\end{aligned}
	\end{equation}
	and for any $s_1>s_2\in[a,\tau_1]$,  by applying the mean value theorem twice and assumption \textbf{A2}, we obtain
	\begin{equation}
		\begin{aligned}
			&\|\sigma(y_{s_1})-\sigma(\tilde{y}_{s_1})-\sigma(y_{s_2})+\sigma(\tilde{y}_{s_2})\|\\
			&\leq C_{\sigma}\ltn y-\tilde{y}\rtn_{\alpha,[a,\tau_1]}(s_1-s_2)^{\alpha}\\
			&~~~+C_{\sigma}\left(\ltn y\rtn_{\alpha,[a,\tau_1]}+\ltn \tilde{y}\rtn_{\alpha,[a,\tau_1]}\right)\ltn y-\tilde{y}\rtn_{\alpha,
				[a,\tau_1]}(s_1-s_2)^{\alpha}(\tau_1-a)^{\alpha},
		\end{aligned}
	\end{equation}
	this  inequality shows that
	\begin{equation}\label{Sec3:3.17}
		\begin{aligned}
			\ltn \sigma(y)-\sigma(\tilde{y})\rtn_{\alpha,[a,\tau_1]}\leq C_\sigma \ltn y-\tilde{y} \rtn_{\alpha,[a,\tau_1]}+2C_{\sigma}\ltn y-\tilde{y}\rtn_{\alpha,[a,\tau_1]}(\tau_1-a)^{\alpha}.
		\end{aligned}
	\end{equation}
	In view of \eqref{Sec3:3.14}-\eqref{Sec3:3.17}, we know that
	\begin{equation}
		\begin{aligned}
			\ltn \mathcal{M}(Y)-\mathcal{M}(\tilde{Y})\rtn_{\alpha,[a,\tau_1]}&\leq 2CC_\sigma\ltn y-\tilde{y} \rtn_{\alpha,[a,\tau_1]}(\ltn W\rtn_{\alpha,[a,\tau_1]}+(\tau_1-a)^{\alpha})\\
			&\leq \frac{1}{2}\ltn y-\tilde{y} \rtn_{\alpha,[a,\tau_1]},
		\end{aligned}
	\end{equation}
	it shows that the mapping $\mathcal{M}$ is contractive.  So the Banach fixed-point theorem can be applied, the 	pure Young differential equation \eqref{Sec1:pure Young} has a unique solution on $C^{\alpha}([\tau_0,\tau_1];\mathbb{R}^m)$ for almost every  fixed $W\in\Omega$. Finally, by repeatedly use Banach fixed-point theorem on each stopping times interval $[\tau_{i},\tau_{i+1}],i=1,\cdots, N_{\gamma,I,\alpha}(x)-1$, we obtain the global solution in $C^{\alpha}(I;\mathbb{R}^m)$.
\end{proof}
\begin{lem}\label{Sec2:lemma-bounded}
	Let $y_t\in C^{\alpha}([a,b];\mathbb{R}^m)$ be the solution to the Young differential equation \eqref{Sec1:pure Young}, and suppose assumption \textbf{A2} holds. Then $\ltn y\rtn_{p-var,[a,b]}\leq N_{\gamma_1,[a,b],p}(W)$ holds for almost every $W\in\Omega$, where $\gamma_1=\frac{1}{4C_{\sigma}C}$.
\end{lem}
\begin{proof}
	For any $s<t\in[a,b]$ and almost every given $W\in\Omega$, we note that
	\begin{equation}
		\begin{aligned}
			y_{s,t}=\int_{s}^{t}\sigma(y_r)dW_r.
		\end{aligned}
	\end{equation}
	By  using the Young bound \eqref{Sec2:p-Youngbounded} and assumption \textbf{A2}, we obtain
	\begin{equation}\nonumber
		\begin{aligned}
			\|y_{s,t}\|&=\left\|\int_{s}^{t}\sigma(y_r)dW_r\right\|\\
			&\leq \left\|\sigma(y_s)W_{s,t}\right\|+C\ltn \sigma(y)\rtn_{p-var,[s,t]}\ltn W\rtn_{p-var,[s,t]}\\
			&\leq C_\sigma\ltn W\rtn_{p-var,[s,t]}+ C\biggr(\sup_{\mathcal{P}([s,t])}\sum_{[u,v]\in\mathcal{P}(s,t)}(\|\sigma(y_v)-\sigma(y_u)\|)^{p}\biggr)^{\frac{1}{p}}\ltn W\rtn_{p-var,[s,t]}\\
			&\leq  C_\sigma\ltn W\rtn_{p-var,[s,t]}+CC_\sigma\biggr(\sup_{\mathcal{P}([s,t])}\sum_{[u,v]\in\mathcal{P}(s,t)}\|y_v-y_u\|^{p}\biggr)^{\frac{1}{p}}\ltn W\rtn_{p-var,[s,t]} \\
			&\leq C_\sigma\ltn W\rtn_{p-var,[s,t]}+CC_\sigma\ltn y\rtn_{p-var,[s,t]}\ltn W\rtn_{p-var,[s,t]}.
		\end{aligned}
	\end{equation}
	According to the definition of $p$--variation semi-norm, the above inequality means that
	$$\ltn y\rtn_{p-var,[s,t]}\leq 2C_{\sigma} \ltn W\rtn_{p-var,[s,t]}+ 2C_{\sigma}C\ltn W\rtn_{p-var,[s,t]}\ltn y\rtn_{p-var,[s,t]}, \forall s,t\in[a,b].$$
	Now, we consider the greedy stopping times sequence \eqref{Sec2:pstopping times} with $\gamma=\gamma_1=\frac{1}{4C_{\sigma}C}$.  On each stopping times interval $[\tau_i,\tau_{i+1}],i=0,\cdots, N_{\gamma_1,[a,b],p}(W)-1$, we have
	$$\ltn y\rtn_{p-var,[\tau_i,\tau_{i+1}]}\leq 1, i=0,\cdots, N_{\gamma_1,[a,b],p}(W)-1.$$
	By Lemma \ref{Sec2: equvalence Control}, it follows that
	\begin{equation}
		\begin{aligned}
			\ltn y\rtn_{p-var,[a,b]}&\leq N_{\gamma_1,[a,b],p}^{\frac{p-1}{p}}(W)\left(\sum_{i=0}^{N_{\gamma_1,[a,b],p}(W)-1}\ltn y\rtn_{p-var,[\tau_i,\tau_{i+1}]}^p\right)^{\frac{1}{p}}\\
			&\leq N_{\gamma_1,[a,b],p}(W).
		\end{aligned}
	\end{equation}
\end{proof}
Furthermore, the solution to equations \eqref{Sec1:pure Young}   depends continuously  on  the initial data.
\begin{lem}\label{Sec3:Lemma 3.2}
Let $y_t$ and $\bar{y}_t$ be the solutions to the Young differential equation \eqref{Sec1:pure Young} with  initial data $y_a$ and $\bar{y}_a$, respectively.  If assumption \textbf{A2} holds.  Then $$\ltn \bar{y}-
	y\rtn_{p-var,[a,b]}\leq 2^{N_{\frac{1}{2\tilde{M}},[a,b],p}-1}N_{\frac{1}{2\tilde{M}},[a,b],p}(\bar{y}_a-y_a)$$ and $$\|\bar{y}-
	y\|_{\infty,[a,b]}\leq( 2^{N_{\frac{1}{2\tilde{M}},[a,b],p}-1}N_{\frac{1}{2\tilde{M}},[a,b],p}+1)(\bar{y}_a-y_a)$$ hold for almost every $W\in\Omega$,  where $\tilde{M}=8C_{\sigma}C(\ltn \bar{y}\rtn_{p-var,[a,b]}+\ltn y\rtn_{p-var,[a,b]}+1)$.	
\end{lem}
\begin{proof}
	For almost every given $W\in\Omega$, observe that
	\begin{equation}
		\begin{aligned}
			\bar{y}_t-y_t=\bar{y}_a-y_a+\int_{a}^{t}\sigma(\bar{y}_r)-\sigma(y_r)dW_r,\quad \forall t\in[a,b].
		\end{aligned}
	\end{equation}
	First, we analyze the $p$--variation semi-norm $\ltn \bar{y}-y\rtn_{p-var,[a,b]}$.  For any $s<t\in[a,b]$,   the Young bound \eqref{Sec2:p-Youngbounded} and assumption \textbf{A2} yield 
	\begin{equation}\label{Sec3:p-var3.22}
		\begin{aligned}
			\|(\bar{y}-y)_{s,t}\|&=\left\|\int_{s}^t\sigma(\bar{y}_r)-\sigma(y_r)dW_r\right\|\\
			&\leq \left\|(\sigma(\bar{y}_s)-\sigma(y_s))W_{s,t}\right\|\!+\!C\ltn \sigma(\bar{y})-\sigma(y)\rtn_{p-var,[s,t]}\ltn W\rtn_{p-var,[s,t]}\\
			&\leq C_\sigma\|\bar{y}_s\!-\!y_s\|\ltn W\rtn_{p-var,[s,t]}\!+\!C\ltn \sigma(\bar{y})\!-\!\sigma(y)\rtn_{p-var,[s,t]}\ltn W\rtn_{p-var,[s,t]}.
		\end{aligned}
	\end{equation}
	For $\ltn \sigma(\bar{y})\!-\!\sigma(y)\rtn_{p-var,[s,t]}$, applying the mean value theorem twice and  assumption \textbf{A2}, we obtain
	\begin{align*}
		&\ltn \sigma(\bar{y})\!-\!\sigma(y)\rtn_{p-var,[s,t]}=\left(\sup_{\mathcal{P}([s,t])}\sum_{[u,v]\in\mathcal{P}(s,t)}\|\sigma(\bar{y}_v)-\sigma(y_v)-\sigma(\bar{y}_u)+\sigma(y_u)\|^p\right)^{\frac{1}{p}}\\
		&~~\leq \biggr(\sup_ {\mathcal{P}([s,t])}\sum_{[u,v]\in\mathcal{P}([s,t])}\biggr(C_\sigma\|\bar{y}_v-y_v-\bar{y}_u+y_u\|\\
		&~~~~~~~~~~~~~~~~~~~~~~~~~~~~~~~~+C_\sigma\left(\|\bar{y}_v-\bar{y}_u\|+\|y_v-y_u\|\right)\|\bar{y}_u-y_u\|\biggr)^p\biggr)^{\frac{1}{p}}\\
		&~~\leq 2C_\sigma\ltn \bar{y}-y\rtn_{p-var,[s,t]} +4C_\sigma(\ltn \bar{y}\rtn_{p-var,[s,t]}+\ltn y\rtn_{p-var,[s,t]})\|\bar{y}-y\|_{\infty,[s,t]}.
	\end{align*}
	Let $\tilde{M}:=8C_{\sigma}C(\ltn \bar{y}\rtn_{p-var,[a,b]}+\ltn y\rtn_{p-var,[a,b]}+1)$, By Lemma \ref{Sec2:lemma-bounded}, it follows that  $\tilde{M}\leq 8C_{\sigma}C(2N_{\gamma_1,[a,b],p}(W)+1)<\infty$ holds for almost every $W\in\Omega$. From  \eqref{Sec3:p-var3.22}, we deduce that
	\begin{equation}
		\begin{aligned}
			\ltn \bar{y}-y\rtn_{p-var,[s,t]}\leq \tilde{M}\|\bar{y}_s-y_s\|\ltn W\rtn_{p-var,[s,t]}+\tilde{M}\ltn \bar{y}-y\rtn_{p-var,[s,t]}\ltn W\rtn_{p-var,[s,t]}
		\end{aligned}
	\end{equation}
	holds for all $[s,t]\subset[a,b]$. Consider the greedy stopping times sequence \eqref{Sec2:pstopping times} with $\gamma=\frac{1}{2\tilde{M}}$, then
	$$\ltn \bar{y}-y\rtn_{p-var,[\tau_i,\tau_{i+1}]}\leq \|\bar{y}_{\tau_{i}}-y_{\tau_i}\|,\quad i=0,\cdots, N_{\frac{1}{2\tilde{M}},[a,b],p}-1.$$
	Therefore, we have
	\begin{equation}\label{Sec:p-var3.24}
		\begin{aligned}
			\ltn \bar{y}-y\rtn_{p-var,[\tau_i,\tau_{i+1}]}&\leq \|\bar{y}_{\tau_{i}}-y_{\tau_i}\|\\
			&\leq  \ltn \bar{y}-y\rtn_{p-var,[\tau_{i-1},\tau_{i}]}+\|\bar{y}_{\tau_{i-1}}-y_{\tau_{i-1}}\|\\
			&\leq 2\|\bar{y}_{\tau_{i-1}}-y_{\tau_{i-1}}\|\\
			&\vdots\\
			&\leq 2^i\|\bar{y}_a-y_a\|.
		\end{aligned}
	\end{equation}
	From \eqref{Sec:p-var3.24} and Lemma \ref{Sec2: equvalence Control},  we derive that
	$$\ltn \bar{y}-y\rtn_{p-var,[a,b]}\leq 2^{N_{\frac{1}{2\tilde{M}},[a,b],p}-1}N_{\frac{1}{2\tilde{M}},[a,b],p}\|\bar{y}_a-y_a\|.$$
	Furthermore, it follows that  $\ltn \bar{y}-y\rtn_{\infty,[a,b]}\leq ( 2^{N_{\frac{1}{2\tilde{M}},[a,b],p}-1}N_{\frac{1}{2\tilde{M}},[a,b],p}+1)\|\bar{y}_a-y_a\|$.
\end{proof}
\begin{thm}\label{Sec3:thm3.2}
	For almost every given $W\in\Omega$.  The solution  $y_t(W,y_a)$ of \eqref{Sec1:pure Young} is differentiable with respect to $y_a$, its derivative $\frac{\partial y_{t}(W,\cdot)}{\partial y}\in C^{\alpha}([a,b];\mathbb{R}^{m\times m})$ is  the matrix-valued  solution of  Young differential equation 
	\begin{equation}\label{Sec3:3.19}
		d\xi_t=D\sigma(y_t)\xi_tdW_t,~ t\in[a,b].
	\end{equation}
\end{thm}
\begin{proof}
	First, we establish the well-posedness of \eqref{Sec3:3.19}. Following a similar approach to the well-posedness of \eqref{Sec1:pure Young}, we employ the Banach fixed-point theorem and concatenation technique. For almost every given $W\in\Omega$, consider a sequence of greedy stopping times $\{\hat{\tau}_{i}(\frac{1}{2M_1},[a,b],\alpha)\},i=0,\cdots ,N_{\frac{1}{2M_1},[a,b],\alpha}$, where $M_1:=2C_\sigma(\ltn y\rtn_{\alpha,[a,b]}+1)C\vee 1$, then we define the solution space
	$$\hat{S}:=\left\{\xi\in C^{\alpha}([a,\hat\tau_1];\mathbb{R}^{m\times m}):\text{with initial data}~\xi_a,~\text{and}~\ltn \xi\rtn_{\alpha,[a,\hat\tau_1]}\leq 2M_1\|\xi_a\|\right\}.$$
	Define the map $\hat{\mathcal{M}}$ as 
	\begin{equation}
		\hat{\mathcal{M}}(\xi)(r):=\xi_a+\int_{a}^{r}D\sigma(y_s)\xi_sdW_s,\quad \xi\in\hat{S},~ r\in [a,\hat{\tau}_1].
	\end{equation}

	\textbf{\textit{Invariance:}} For any $\xi\in \hat{S}$ and $s<r\in[a,\hat{\tau}_1]$,  the estimate  in  Remark \ref{Sec2:Young est-H} yields
	\begin{equation}\label{Sec3:3.21}
		\begin{aligned}
			\| \hat{\mathcal{M}}(\xi)_{s,r}\|&=\left\|\int_{s}^{r}D\sigma(y_{s^\prime})\xi_{s^\prime}dW_{s^\prime}\right \|\\
			&\leq \left\|D\sigma(y_s)\xi_sW_{s,r}\right\|+C\ltn D\sigma(y)\xi\rtn_{\alpha,[s,r]}\ltn W \rtn_{\alpha,[s,r]}(r-s)^{2\alpha}\\
			&\leq \|D\sigma(y_s)\|\left\|\xi_s\right\|\ltn W\rtn_{\alpha,[s,r]}(r-s)^{\alpha}\\
			&~~~~+C\ltn D\sigma(y)\xi\rtn_{\alpha,[s,r]}\ltn W \rtn_{\alpha,[s,r]}(r-s)^{2\alpha}.
		\end{aligned}
	\end{equation}
	Note that
	\begin{equation}\label{Sec3:3.22}
		\begin{aligned}
			\|D\sigma(y_s)\|\left\|\xi_s\right\|&\leq C_\sigma\|\xi_s\|\\
			&\leq C_\sigma\ltn \xi \rtn_{\alpha,[a,\hat{\tau}_1]}(\hat{\tau}_1-a)^{\alpha}+C_\sigma\|\xi_a\|,
		\end{aligned}
	\end{equation}
	and for any $s_1>s_2\in [s,r]$, we have
	\begin{equation}\label{Sec3:3.23}
		\begin{aligned}
			&\left\|D\sigma(y_{s_1})\xi_{s_1}-D\sigma(y_{s_2})\xi_{s_2}\right\|\\
			&\leq \|D\sigma(y_{s_1})-D\sigma(y_{s_2})\|\left\|\xi_{s_1}\right\|+\|D\sigma(y_{s_2})\|\left\| \xi_{s_1}-\xi_{s_2}\right\|\\
			&\leq C_{\sigma}\|y_{s_1}-y_{s_2}\|\|\xi_{s_1}\|+C_\sigma\|\xi_{s_1}-\xi_{s_2}\|\\
			&\leq \left[C_{\sigma}\ltn y\rtn_{\alpha,[a,\hat{\tau}_1]}(\ltn \xi\rtn_{\alpha,[a,\hat{\tau}_1]}(\hat{\tau}_1-a)^{\alpha}+\|\xi_a\|)+C_\sigma\ltn \xi\rtn_{\alpha,[a,\hat{\tau}_1]}\right](s_1-s_2)^\alpha.
		\end{aligned}
	\end{equation}
	So \eqref{Sec3:3.23} shows that
	\begin{equation}\label{Sec3:3.24}
		\begin{aligned}
			\ltn D\sigma(y)\xi\rtn_{\alpha,[s,r]}&\leq C_\sigma\ltn y\rtn_{\alpha,[a,\hat{\tau}_1]}(\|\xi_a\|+\ltn \xi\rtn_{\alpha,[a,\hat{\tau}_1]}(\hat{\tau}_1-a)^{\alpha})+C_\sigma\ltn \xi\rtn_{\alpha,[a,\hat{\tau}_1]}.
		\end{aligned}
	\end{equation}
	From \eqref{Sec3:3.21},\eqref{Sec3:3.22},\eqref{Sec3:3.24}, we know that
	\begin{equation}
		\begin{aligned}
			\ltn \hat{\mathcal{M}}(\xi)\rtn_{\alpha,[a,\hat{\tau}_1]}\leq& C_\sigma C(\ltn y\rtn_{\alpha,[a,b]}+1)\|\xi_a\|\\
			&+2C_\sigma C(\ltn y\rtn_{\alpha,[a,b]}+1)\ltn \xi\rtn_{\alpha,[a,\hat{\tau}_1]}(\ltn W\rtn_{\alpha,[a,\hat{\tau}_1]}+(\hat{\tau}_1-a)^{\alpha})\\
			\leq&M_1\|\xi_a\|+M_1\ltn \xi \rtn_{\alpha,[a,\hat{\tau}_1]}(\ltn W\rtn_{\alpha,[a,\hat{\tau}_1]}+(\hat{\tau}_1-a)^{\alpha})\\
			\leq &2M_1\|\xi_a\|.
		\end{aligned}
	\end{equation}
	Thus, the mapping $\hat{\mathcal{M}}$ map $\hat{S}$ to $\hat{S}$.
	
	\textbf{\textit{Contraction:}} For any $\xi,\theta\in\hat{S}$, 
	\begin{equation}
		\begin{aligned}
			\hat{\mathcal{M}}(\xi)(t)\!-\!\hat{\mathcal{M}}(\theta)(t)=\int_{a}^{t}D\sigma(y_s)(\xi_s-\theta_s)dW_s,\quad t\in [a,\hat{\tau}_1].
		\end{aligned}
	\end{equation}
	For any $s_1<s_2\in[a,\hat{\tau}_1]$,  according  to  Remark \ref{Sec2:Young est-H} we have
	\begin{equation}\label{Sec3:3.31}
		\begin{aligned}
			&\left\| (\hat{\mathcal{M}}(\xi)-\hat{\mathcal{M}}(\theta))_{s_1,s_2}\right\|=\left\|\int_{s_1}^{s_2}D\sigma(y_s)(\xi_s-\theta_s)dW_s\right\|\\
			&\leq\left\|D\sigma(y_{s_1})(\xi_{s_{1}}-\theta_{s_{1}})W_{s_1,s_2}\right\|\\
			&~~~~+C\ltn D\sigma(y)(\xi-\theta)\rtn_{\alpha,[s_1,s_2]}\ltn W\rtn_{\alpha,[s_1,s_2]}(s_2-s_1)^{2\alpha}.
		\end{aligned}
	\end{equation}
	Similar to the computation of  \eqref{Sec3:3.22}, we have
	\begin{equation}\label{Sec3:3.27}
		\begin{aligned}
			&\left\|D\sigma(y_{s_1})(\xi_{s_1}-\theta_{s_1})W_{s_1,s_2}\right\|\\
			&~~~~\leq C_\sigma\ltn \xi-\theta\rtn_{\alpha,[a,\hat{\tau}_1]}\ltn W\rtn_{\alpha,[a,\hat{\tau}_1]}(\hat{\tau}_1-a)^{\alpha}(s_2-s_1)^{\alpha},
		\end{aligned}
	\end{equation}
	and similar to the computations of \eqref{Sec3:3.23} and \eqref{Sec3:3.24},
	\begin{equation}\label{Sec3:3.29}
		\begin{aligned}
			&\ltn D\sigma(y)(\xi-\theta) \rtn_{\alpha,[s_1,s_2]}\\
			&\leq C_\sigma\ltn \xi-\theta\rtn_{\alpha,[a,\hat{\tau}_1]}+C_\sigma\ltn y\rtn_{\alpha,[a,b]}\ltn \xi-\theta \rtn_{\alpha,[a,\hat{\tau}_1]} (\hat{\tau}_1-a)^{\alpha}.
		\end{aligned}
	\end{equation}
	\eqref{Sec3:3.31}-\eqref{Sec3:3.29} imply that
	\begin{equation}
		\begin{aligned}
			\ltn \hat{\mathcal{M}}(\xi)-\hat{\mathcal{M}}(\theta) \rtn_{\alpha,[a,\hat{\tau}_1]} &\leq 2C_{\sigma}C(\ltn y\rtn_{\alpha,[a,b]}+1)\ltn \xi-\theta\rtn_{\alpha,[a,\hat{\tau}_1]}(\ltn W\rtn_{\alpha,[a,\hat{\tau}_1]}+(\hat{\tau}_1-a)^{\alpha} )\\
			&\leq \frac{1}{2}\ltn \xi-\theta\rtn_{\alpha,[a,\hat{\tau}_1]}.
		\end{aligned}
	\end{equation}
	Thus, the mapping $\hat{\mathcal{M}}$ is contractive on $\hat{S}$. The Banach fixed-point theorem shows that equations \eqref{Sec3:3.19} has a unique $\alpha$-H\"{o}lder continuous solution on $[a,\hat{\tau}_1]$.  As the proof of Theorem \ref{Sec3:Thm3.1},  by repeatedly using  the Banach fixed-point theorem on each stopping times interval $[\hat{\tau}_{i},\hat{\tau}_{i+1}],i=1,\cdots, N_{\gamma,I,\alpha}(x)-1$, we  can obtain the global solution in $C^{\alpha}([a,b];\mathbb{R}^{m\times m})$.
	
For alomst every given $W$. 	Next, we show that the solution of \eqref{Sec3:3.19} is Fr\'{e}chet derivative of  $y_t(W,y_a)$ with respect to $y_a$.  Let  $\Phi(t,W,y_a)=\xi_t$ be	 the solution operator of \eqref{Sec3:3.19}. Let $\eta_t:=\Phi(t,W,y_a)(\bar{y}_a-y_a)$ and $r_t=\bar{y}_t-y_t-\eta_t$, then $r_a=0$.  It follows that 
	\begin{equation}\nonumber
		\begin{aligned}
			r_t&=\int_{a}^{t}\sigma(\bar{y}_r)-\sigma(y_r)-D\sigma(y_r)\xi_r(\bar{y}_a-y_a)dW_r\\
			&=\int_{a}^{t}\int_{0}^{1}[D\sigma(y_r+\theta(\bar{y}_r-y_r))-D\sigma(y_r)](\bar{y}_r-y_r)d\theta dW_r\\
			&~~~~~+\int_{a}^{t}D\sigma(y_s)r_sdW_s.
		\end{aligned}
	\end{equation}
	Now, we  estimate $\|r\|_{\infty,[a,b]}$ and $\ltn r\rtn_{p-var,[a,b]}$. From above equality, we have  
	\begin{equation}
		\begin{aligned}
			r_{s,t}&=e_{s,t}+\int_{s}^{t}D\sigma(y_{s^\prime})r_{s^\prime}dW_{s^\prime},~s<t\in[a,b]
		\end{aligned}
	\end{equation}
	where $e_{s,t}=\int_{s}^{t}\int_{0}^{1}[D\sigma(y_r+\theta(\bar{y}_r-y_r))-D\sigma(y_r)](\bar{y}_r-y_r)d\theta dW_r$. Then,  Young bound \eqref{Sec2:p-Youngbounded} yields
	\begin{equation}
		\begin{aligned}
			\|r_{s,t}\|&\leq \|e_{s,t}\|+\left\|\int_{s}^{t}D\sigma(y_{s^\prime})r_{s^\prime}dW_{s^\prime}\right\|\\
			&\leq \|e_{s,t}\|+\left\|D\sigma(y_{s})r_sW_{s,t}\right\|+C\ltn D\sigma(y)r \rtn_{p-var,[s,t]}\ltn W\rtn_{p-var,[s,t]}\\
			&\leq \|e_{s,t}\|+C_{\sigma}\|r_s\|\ltn W\rtn_{p-var,[s,t]}+C\ltn D\sigma(y)r\rtn_{p-var,[s,t]}\ltn W\rtn_{p-var,[s,t]}.
		\end{aligned}
	\end{equation}
	Observe that
	\begin{equation}\nonumber
		\begin{aligned}
			\ltn D\sigma(y)r \rtn_{p-var,[s,t]}&=\biggr(\sup_{\mathcal{P}([s,t])}\sum_{[u,v]\in\mathcal{P}([s,t])}\biggr\|D\sigma(y_v)r_v
			-D\sigma(y_u)r_u \biggr\|^p\biggr)^{\frac{1}{p}}\\
			&\leq 2\biggr(\sup_{\mathcal{P}([s,t])}\sum_{[u,v]\in\mathcal{P}([s,t])}\biggr\|(D\sigma(y_v)
			-D\sigma(y_u))r_v\biggr\|^p\biggr)^{\frac{1}{p}}\\
			&~~~~+2\biggr(\sup_{\mathcal{P}([s,t])}\sum_{[u,v]\in\mathcal{P}([s,t])}\biggr\|
			D\sigma(y_u)r_{u,v} \biggr\|^p\biggr)^{\frac{1}{p}}\\
			&\leq 2C_\sigma\ltn y\rtn_{p-var,[s,t]}\|r\|_{\infty,[s,t]}+2C_\sigma\ltn r\rtn_{p-var,[s,t]}.
		\end{aligned}
	\end{equation}
	Hence,
	\begin{equation}\label{Sec3:3.39}
		\begin{aligned}
			\|r_{s,t}\|\leq &\|e_{s,t}\|+2CC_\sigma\ltn W\rtn_{p-var,[s,t]}\|r\|_{p-var,[s,t]}\\
			&+2C_\sigma C(\ltn y\rtn_{p-var,[a,b]}+1)\ltn W\rtn_{p-var,[s,t]}\|r\|_{\infty,[s,t]}.
		\end{aligned}
	\end{equation}
	Then, we need to estimate $\|e_{s,t}\|$, using  Young bound \eqref{Sec2:p-Youngbounded} and  mean-value theorem twice, one gets
	\begin{equation}\nonumber
		\begin{aligned}
			\|e_{s,t}\|&\leq \left\|\int_{0}^{1}[D\sigma(y_s+\theta(\bar{y}_s-y_s))-D\sigma(y_s)](\bar{y}_s-y_s)d\theta W_{s,t}\right\|\\
			&~~~~+C\ltn\int_{0}^{1}[D\sigma(y+\theta(\bar{y}-y))-D\sigma(y)](\bar{y}-y)d\theta \rtn_{p-var,[s,t]}\ltn W\rtn_{p-var,[s,t]}\\
			&\leq C_\sigma\|\bar{y}-y\|^2_{\infty,[a,b]}\ltn W\rtn_{p-var,[s,t]}+8CC_\sigma\ltn \bar{y}-y\rtn_{p-var,[a,b]}\|\bar{y}-y\|_{\infty,[a,b]}\ltn W\rtn_{p-var,[s,t]}\\
			&~~~~+8CC_\sigma(\ltn \bar{y}\rtn_{p-var,[a,b]}+\ltn y\rtn_{p-var,[a,b]})\|\bar{y}-y\|_{\infty,[a,b]}^2\ltn W\rtn_{p-var,[s,t]}\\
			&\leq 8CC_\sigma(\ltn \bar{y}\rtn_{p-var,[a,b]}+\ltn y\rtn_{p-var,[a,b]}+1)\|\bar{y}-y\|_{\infty,[a,b]}^2\ltn W\rtn_{p-var,[s,t]}\\
			&~~~~+8CC_\sigma\ltn \bar{y}-y\rtn_{p-var,[a,b]}\|\bar{y}-y\|_{\infty,[a,b]}\ltn W\rtn_{p-var,[s,t]}.
		\end{aligned}
	\end{equation}
	The above inequality shows that
	\begin{equation}\label{Sec3:3.40}
		\begin{aligned}
			\ltn e\rtn_{p-var,[s,t]}&\leq 16CC_\sigma(\ltn \bar{y}\rtn_{p-var,[a,b]}+\ltn y\rtn_{p-var,[a,b]}+1)\|\bar{y}-y\|_{\infty,[a,b]}^2\ltn W\rtn_{p-var,[s,t]}\\
			&~~~~+16CC_\sigma\ltn \bar{y}-y\rtn_{p-var,[a,b]}\|\bar{y}-y\|_{\infty,[a,b]}\ltn W\rtn_{p-var,[s,t]}.
		\end{aligned}
	\end{equation}
For any $s<t\in[a,b]$, 	in view of  \eqref{Sec3:3.39} and \eqref{Sec3:3.40}, we have that
	\begin{equation}
		\begin{aligned}
			\ltn r \rtn_{p-var,[s,t]}&\leq 64CC_\sigma(\ltn \bar{y}\rtn_{p-var,[a,b]}+\ltn y\rtn_{p-var,[a,b]}+1)\|\bar{y}-y\|_{\infty,[a,b]}^2\ltn W\rtn_{p-var,[s,t]}\\
			&~~~~+64CC_\sigma\ltn \bar{y}-y\rtn_{p-var,[a,b]}\|\bar{y}-y\|_{\infty,[a,b]}\ltn W\rtn_{p-var,[s,t]}\\
			&~~~~+16CC_\sigma(\ltn y\rtn_{p-var,[a,b]}+1)\ltn W\rtn_{p-var,[s,t]}\|r\|_{p-var,[s,t]}\\
			&~~~~+8CC_\sigma(\ltn y\rtn_{p-var,[a,b]}+1)\ltn W\rtn_{p-var,[s,t]}\|r_s\|.
		\end{aligned}
	\end{equation}
	Consider  the greedy stopping times sequence \eqref{Sec2:pstopping times} with $\gamma=\gamma_2=\frac{1}{32CC_\sigma(\ltn \bar{y}\rtn_{p-var,[a,b]}+\ltn y\rtn_{p-var,[a,b]}+1)}$, then for any $i\in\{0,1,\cdots,N_{\gamma,[a,b],p}-1\}$, we obtain
	\begin{equation}
		\begin{aligned}
			\ltn r\rtn_{p-var,[\tau_i,\tau_{i+1}]}&\leq \|r_{\tau_i}\|+4\|\bar{y}-y\|_{\infty,[a,b]}(\|\bar{y}-y\|_{\infty,[a,b]}+\ltn \bar{y}-y\rtn_{p-var,[a,b]})\\
			&\leq \|r_a\|+4(i+1)\|\bar{y}-y\|_{\infty,[a,b]}(\|\bar{y}-y\|_{\infty,[a,b]}+\ltn \bar{y}-y\rtn_{p-var,[a,b]}).
		\end{aligned}
	\end{equation}
	Hence, by $r_a=0$ and Lemma \ref{Sec2: equvalence Control}, we obtain
	\begin{equation}\label{Sec3:3.43}
		\begin{aligned}
			\|r\|_{\infty,[a,b]}&\leq \ltn r\rtn_{p-var,[a,b]}\\
			& \leq 4\|\bar{y}-y\|_{\infty,[a,b]}(\|\bar{y}-y\|_{\infty,[a,b]}+\ltn \bar{y}-y\rtn_{p-var,[a,b]})N_{\gamma_2,[a,b],p}^2.
		\end{aligned}
	\end{equation}
	By \eqref{Sec3:3.43} and Lemma \ref{Sec3:Lemma 3.2}, there exists a constant $D=D(p,[a,b],\ltn W\rtn_{p-var,[a,b]})$ such that
	\begin{equation}\label{Sec3:1:3.42}
		\|\bar{y}_{t}(W,\bar{y}_a)-y_{t}(W,y_a)-\Phi(t,W,y_a)(\bar{y}_a-y_a)\|\leq D\|\bar{y}_a-y_a\|^2
	\end{equation}
	hold for all $t\in[a,b]$. Then the solution of \eqref{Sec3:3.19} is Fr\'{e}chet derivative of  $y_t(W,y_a)$, we complete the proof.
\end{proof}
\begin{rem}
	For the well-posedness of equations  \eqref{Sec1:pure Young}, $\sigma\in C_b^2$ is a classical assumption. However, $C_b^3$-regularity is necessary to ensure the differentiability of solution flow with respect to  initial data.  
\end{rem}

\begin{rem}
	For simplicity, we use $\varphi(t,W,y_a)$ and $\frac{\partial \varphi(t,W,\cdot)}{\partial y}$ to represent the solution $y_t(W,y_a)$ and its Fr\'{e}chet derivative with initial data, respectively. 
\end{rem}
Similarly, the analogous results hold  for  the  backward direction.
\begin{cor}\label{derivative-back}
	For almost every given $W\in\Omega$, 
	the  solution $h_t(W,h_b)$ of backward Young differential equation \eqref{Sec1:bpure Young} is differentiable with respect to terminal data $h_b$.  The derivative $\frac{\partial h_t(W,h_b)}{\partial h_b}$ satisfies the following  Young differential equation 
		$$d\rho_t=D\sigma(h_t)\rho_t dW_t,\quad t\in[a,b].$$
	Moreover, the estimates in Lemma \ref{Sec3:Lemma 3.2} also hold for $h_t$.
\end{cor}
Similar to the step 2 of  the proof of Theorem \ref{Sec3:thm3.2}, we obtain the following result:
\begin{cor}\label{Sec3:cor3.2}
	For almost every given $W\in\Omega$.  $\frac{\partial \varphi(t,W,\cdot)}{\partial y}$ is globally Lipschitz continuous with respect to initial data $y_a$.
\end{cor}
\begin{proposition}
For almost every given $W\in\Omega$. 	Let $\bar{y}_t$ and $y_t$ be the solutions of \eqref{Sec1:pure Young} with initial data $\bar{y}_a$ and $y_a$, respectively. If arbitrary interval $[a,b]$   satisfying $32C_{\sigma}C\ltn W\rtn_{p-var,[a,b]}\leq 1$, then we have the following estimates
	\begin{equation}\label{Sec3:3.45}
		\begin{aligned}
			\ltn y\rtn_{p-var,[a,b]}&\leq4C_{\sigma}\ltn W\rtn_{p-var,[a,b]}\leq  \frac{1}{2},\\
			\ltn \bar{y}-y\rtn_{p-var,[a,b]}&\leq 16C_{\sigma} C\ltn W\rtn_{p-var,[a,b]}\|\bar{y}_a-y_a\|.
		\end{aligned}
	\end{equation}
\end{proposition}
\begin{proof}
	For any $s<
	t\in[a,b]$,	Young bound \eqref{Sec2:p-Youngbounded} yields 
	\begin{equation}
		\begin{aligned}
			\|y_{s,t}\|&\leq \|\sigma(y_s)W_{s,t}\|+C\ltn \sigma(y)\rtn_{p-var,[s,t]}\ltn W\rtn_{p-var,[s,t]}\\
			&\leq C_{\sigma}\ltn W\rtn_{p-var,[s,t]}+CC_{\sigma}\ltn y\rtn_{p-var,[s,t]}\ltn W\rtn_{p-var,[s,t]}.
		\end{aligned}
	\end{equation}
	By the definition of $p$--variation semi-norm, we obtain
	\begin{equation}
		\begin{aligned}
			\ltn y\rtn_{p-var,[a,b]}&\leq 2C_{\sigma}\ltn W\rtn_{p-var,[a,b]}+2CC_{\sigma}\ltn y\rtn_{p-var,[a,b]}\ltn W\rtn_{p-var,[a,b]}\\
			&\leq 2C_{\sigma}\ltn W\rtn_{p-var,[a,b]}+\frac{1}{2}\ltn y\rtn_{p-var,[a,b]}.
		\end{aligned}
	\end{equation}
	The above inequality implies that $\ltn y\rtn_{p-var,[a,b]}\leq
	4C_{\sigma}\ltn W\rtn_{p-var,[a,b]}\leq \frac{1}{2}.$  Similarly, $\ltn \bar y\rtn_{p-var,[a,b]}\leq \frac{1}{2}.$ 
	
	Observe that
	$$\bar{y}_{s,t}-y_{s,t}=\int_{s}^{t}\sigma(\bar{y}_r)-\sigma(y_r)dW_r.$$
	For any $s<t\in[a,b]$,  using Young bound \eqref{Sec2:p-Youngbounded} and mean value theorem twice  and assumption \textbf{A2}, we have  
	\begin{equation}\nonumber
		\begin{aligned}
			\|\bar{y}_{s,t}-y_{s,t}\|&\leq \left\|\int_{s}^{t}\sigma(\bar{y}_r)-\sigma(y_r)dW_r\right\|\\
			&\leq C_\sigma\ltn W\rtn_{p-var,[s,t]}\|\bar{y}\!-\!y\|_{\infty,[a,b]}\!+\!2C_\sigma C\ltn W\rtn_{p-var,[s,t]}\ltn \bar{y}\!-\!y\rtn_{p-var,[s,t]}\\
			&~~+2C_\sigma C\left(\ltn\bar{y}\rtn_{p-var,[a,b]}+\ltn y\rtn_{p-var,[a,b]}\right)\|\bar{y}-y\|_{\infty,[a,b]}\ltn W \rtn_{p-var,[s,t]}.
		\end{aligned}
	\end{equation}
	The above inequality shows that
	\begin{equation}\nonumber
		\begin{aligned}
			\ltn \bar{y}-y\rtn_{p-var,[a,b]}&\leq 4C_\sigma C\left(\ltn\bar{y}\rtn_{p-var,[a,b]}\!+\!\ltn y\rtn_{p-var,[a,b]}\!+\!1\right)\ltn W\rtn_{p\!-\!var,[a,b]}\|\bar{y}\!-\!y\|_{\infty,[a,b]}\\
			&~~+4C_\sigma C\ltn W\rtn_{p-var,[a,b]}\|\bar{y}\!-\!y\|_{p-var,[a,b]}\\
			&\leq 16C_\sigma C\ltn W\rtn_{p-var,[a,b]}\ltn \bar{y}-y\rtn_{p-var,[a,b]}\\
			&~~+8C_\sigma C\ltn W\rtn_{p-var,[a,b]}\|\bar{y}_a-y_a\|.
		\end{aligned}
	\end{equation}
	Therefore,
	\begin{equation}
		\begin{aligned}
			\ltn \bar{y}-y\rtn_{p-var,[a,b]}\leq  16C_{\sigma}C\ltn W\rtn_{p-var,[a,b]}\|\bar{y}_a-y_a\|.
		\end{aligned}
	\end{equation}
\end{proof}
\begin{rem}\label{Sec3:Rem 3.2}
For almost every $W\in\Omega$,	it can be shown  that the inequalities of \eqref{Sec3:3.45}  also hold for the backward equation \eqref{Sec1:bpure Young}. In addition,  we use $\psi(t,W,h_b)$ and $\frac{\partial \psi(t,W,\cdot)}{\partial h}$ to represent the solution $h_t(W,h_b)$ and its Fr\'{e}chet derivative with terminal data, respectively. 
\end{rem}
For almost every given $W\in\Omega$. We note that
$$\varphi(t,W,\cdot)\circ\psi(a,W,y_t)=y_t,$$
where $y_t=\varphi(t,W,y_a)$. Theorem \ref{Sec3:thm3.2} and Corollary \ref{derivative-back} imply that $\varphi$ and $\psi$ are differentiable with respect to initial data and terminal data, respectively. Furthermore,
\begin{equation}\label{Sec3:1:3.48}
	\frac{\partial \varphi(t,W,y_a)}{\partial y}\frac{\partial \psi(a,W,y_t)}{\partial h}=Id.
\end{equation}
holds for almost every given $W\in\Omega$, where $Id$ is a identity operator.

Based on the analysis, we  can give the proof of Theorem \ref{Sec:thm 3.1}.
\begin{proof}
	In order to get the global well-posedness of equations \eqref{Sec1:equi-eq}, we employ the Doss-Sussmann transformation to transform the stochastic system into a random differential equation. To this end, we need  the well-posedness of \eqref{Sec1:zeq}. For almost every given $W\in\Omega$,  let interval $[a,b]$ satisfy $32C_{\sigma}C\ltn W\rtn_{p-var,[a,b]}\leq \lambda$ for some $\lambda\in(0,1)$.  Similar to Corollary \ref{Sec3:cor3.2}, $\frac{\partial\psi}{\partial h}(t,W,\cdot)$ is  globally Lipschitz, and by Theorem \ref{Sec3:thm3.2}, $\varphi(t,W,\cdot)$  is  also globally Lipschitz. In addition, since $f$ and $g$ are locally Lipschitz, the right side of \eqref{Sec1:zeq} is locally Lipschitz. To handle the local Lipschitz property, we introduce the truncated functions  $f_R$ and $g_R$, which is coincide with $f$ and $g$ on a ball $B_R(0)$ with center $0$ and radius $R$ and vanish outside $B_{R+1}(0)$.   Using the truncated functions  $f_R,g_R$ to replace $f$ and $g$ in equation \eqref{Sec1:zeq}.  Applying the Banach fixed-point theorem, the truncated version of \eqref{Sec1:zeq} admits a unique solution in  $C([a,b];\mathbb{R}^m\times\mathbb{R}^m)$.

	From the first inequality of \eqref{Sec3:3.45}, 
	\begin{equation}\label{SecA:A40}
	\|\varphi(t,W,Z_t^i)-Z_t^i\|\leq\ltn \varphi(\cdot,W,Z_t^i)\rtn_{p-var,[a,b]}\leq 4C_\sigma\ltn W\rtn_{p-var,[a,b]}\leq \frac{\lambda}{2}
	\end{equation}
	holds for almost all $W\in\Omega$ and $i\in\{1,2\}$.
	
	For  $i=1,2$ and almost every $W\in\Omega$, by  Remark \ref{Sec3:Rem 3.2}, $\tilde{\psi}$, defined in Sec. \ref{Sec:synchronization}, satisfy that 
	\begin{equation}\label{SecA:A41}
		\begin{aligned}
			\|\tilde{\psi}(a,W,Z_t^i)\|\leq \ltn \tilde{\psi}(\cdot,W,Z_t^i)\rtn_{p-var,[a,b]}&=\ltn \frac{\partial\psi}{\partial h}(\cdot,W,\varphi(t,W,Z_t^i)) \rtn_{p-var,[a,b]}\\
			&\leq 16C_{\sigma}C\ltn W\rtn_{p-var,[a,b]}\leq \frac{\lambda}{2}.
		\end{aligned}
	\end{equation}
	Rewriting  \eqref{Sec1:zeq} as
	\begin{equation}\label{Sec3:zeq1}
		\left\{
		\begin{aligned}
			\dot{Z}^1_t&= (Id+\tilde{\psi}(a,W,Z^1_t)) f(\varphi(t,W,Z^1_t))+\kappa(Z^2_t-Z^1_t),\\
			\dot{Z}^2_t&= (Id+\tilde{\psi}(a,W,Z^2_t)) g(\varphi(t,W,Z^2_t))+\kappa(Z^1_t-Z^2_t).
		\end{aligned}
		\right.
	\end{equation}
	We claim that there exist some constant $\bar{C}_{\lambda}>0,\delta_{\lambda}>0$ such that
	\begin{equation}\label{Sec3:3.51}
		\begin{aligned}
			\frac{d}{2dt}\|Z_t\|^2\leq \bar{C}_{\lambda}-\delta_{\lambda}\|Z_t\|^2
		\end{aligned}
	\end{equation}
	holds for almost every given $W\in\Omega$. Indeed, using Young's inequality, 
	\begin{equation}
		\begin{aligned}
			\frac{d\|Z_t\|^2}{2dt}&=\frac{d(\|Z^1_t\|^2+\|Z^2_t\|^2)}{2dt}\\
			&\leq \<(Id+\tilde{\psi}(t,W,Z^1_t)) f(\varphi(t,W,Z^1_t)),Z_t^1\>\\
			&~~~~+\<(Id+\tilde{\psi}(t,W,Z^2_t)) g(\varphi(t,W,Z^2_t)),Z^2_t\>
		\end{aligned}
	\end{equation}
holds for almost every given $W\in\Omega$.	For the above inequality, using the same arguments in Ref. \cite[Theorem 2.1]{MR4385780}, we can get \eqref{Sec3:3.51}.  Therefore, the inequality \eqref{Sec3:3.51} shows that
	$$\|Z_t\|\leq \|Y_a\|+\sqrt{\frac{\bar{C}_{\lambda}}{\delta_{\lambda}}}$$
	holds for all  $t\in [a,b]$ and almost every given $W\in\Omega$. This estimate also holds for the truncated system. Thus,    Lemma \ref{Sec2:lemma-bounded} and the above inequality ensure  that the finite  $R$ can be chosen for almost every given $W\in\Omega$,  we obtain the global solution $Z_t$ on such interval $[a,b]$.
	Finally, for general $[a,b]$ and almost every given $W\in\Omega$, if  $32C_{\sigma}C\ltn W\rtn_{p-var,[a,b]}> \lambda$, we partition $[a,b]$ by using the greedy stopping time sequence $\{\tau_{i}(\frac{\lambda}{32C_{\sigma}C},[a,b],p)\}_{i\in\mathbb{N}}$. By repeating the above procedure $N{\frac{\lambda}{32C_{\sigma}C},[a,b],p}(W)$ times, we obtain the unique global solution  $Z_t\in C([a,b];\mathbb{R}^m\times\mathbb{R}^m)$ of almost every given $W\in\Omega$.
	
	For almost every given $W\in\Omega$. Define $Y_t:=(\varphi(t,W,Z^1_t),\varphi(t,W,Z^2_t))$, we will illustrate  that $Y_t$  is a solution of  \eqref{Sec1:equi-eq}. Indeed, for any $s<t\in[a,b]$,
	\begin{equation}\label{Sec3:3.55}
		\begin{aligned}
			Y_{s,t}^1&=\varphi(t,W,Z^1_t)-\varphi(s,W,Z^1_s)\\
			&=\varphi(t,W,Z^1_t)-\varphi(t,W,Z^1_s)+\varphi(t,W,Z^1_s)-\varphi(s,W,Z^1_s)\\
			&=\varphi(t,W,Z^1_t)-\varphi(t,W,Z^1_s)+\int_{s}^{t}\sigma(\varphi(u,W,Z_s^1))dW_u.	
		\end{aligned}
	\end{equation}
	Using \eqref{Sec3:1:3.42} and Theorem \ref{Sec3:thm3.2}, we have
	\begin{equation}\label{Sec3:3.56}
		\begin{aligned}
			&\|\varphi(t,W,Z^1_t)-\varphi(t,W,Z^1_s)-\frac{\partial \varphi(s,W,Z_s^1)}{\partial y}Z_{s,t}^1\|\\
			&\leq \left\|\varphi(t,W,Z^1_t)-\varphi(t,W,Z^1_s)-\frac{\partial \varphi(t,W,Z_s^1)}{\partial y}Z_{s,t}^1\right\|\\
			&~~~~+\left\|\frac{\partial \varphi(t,W,Z_s^1)}{\partial y}Z_{s,t}^1-\frac{\partial \varphi(s,W,Z_s^1)}{\partial y}Z_{s,t}^1 \right\|\\
			&\leq D\|Z_{s,t}^1\|^2+\left\|\int_{s}^{t}D\sigma(\varphi(u,W,Z_s^1))\frac{\partial \varphi(u,W,Z_s^1)}{\partial y}dW_u\right\|\|Z_{s,t}^1\|\\
			&\lesssim (t-s)^{1+\alpha}.
		\end{aligned}
	\end{equation}
	In addition,  due to the globally Lipschitz property of $\frac{\partial \psi(a,W,\cdot)}{\partial h}$, Theorem \ref{Sec3:thm3.2}, Corollary \ref{Sec3:cor3.2} and  \eqref{Sec3:1:3.48}, we have 
	\begin{equation}\label{Sec3:3.57}
		\begin{aligned}
			\frac{\partial \varphi(s,W,Z_s^1)}{\partial y}Z_{s,t}^1&=\frac{\partial \varphi(s,W,Z_s^1)}{\partial y}\frac{\partial \psi(a,W,\varphi(s,W,Z_s^1))}{\partial h}f(\varphi(s,W,Z_s^1))(t-s)\\
			&~~~~+\kappa\frac{\partial \varphi(s,W,Z_s^1)}{\partial y}
			(Z_s^2-Z_s^1)(t-s)\\
			&~~~~+\mathbf{o}((t-s)^2)\\
			&=f(\varphi(s,W,Z_s^1))(t-s)\\
			&~~~~+\kappa\frac{\partial \varphi(s,W,Z_s^1)}{\partial y}
			(Z_s^2-Z_s^1)(t-s)\\
			&~~~~+\mathbf{o}((t-s)^2).\\
		\end{aligned}
	\end{equation}
	Therefore, according to  \eqref{Sec3:3.55}-\eqref{Sec3:3.57}, $Z\in C([a,b],\mathbb{R}^m\times\mathbb{R}^m)$, we have
	\begin{align*}
		Y^1_{s,t}&=\lim_{|\mathcal{P}([s,t])|\rightarrow 0}\sum_{[u,v]\in\mathcal{P}([s,t])}Y^1_{u,v}\\
		&=\lim_{|\mathcal{P}([s,t])|\rightarrow 0}\sum_{[u,v]\in\mathcal{P}([s,t])}[f(\varphi(u,W,Z_u^1))(v-u)\\
		&~~~~+\kappa\frac{\partial \varphi(s,W,Z_s^1)}{\partial y}
		(Z_u^2-Z_u^1)(v-u)]+\lim_{|\mathcal{P}([s,t])|\rightarrow 0}\sum_{[u,v]\in\mathcal{P}([s,t])}\int_{s}^{t}\sigma(\varphi(u,W,Z_s^1))dW_u\\
		&=\int_{s}^{t}f(y_r)+\kappa\frac{\partial \varphi(r,W,Z_r^1)}{\partial y}
		(Z^2_r-Z^1_r)dr+\int_{s}^{t}\sigma(Y^1_r)dW_r.
	\end{align*}
	Similarly, we have
	$$Y^2_{s,t}=\int_{s}^{t}g(Y^2_r)+\kappa\frac{\partial \varphi(r,W^2,Z_r^2)}{\partial y}
	(Z^1_r-Z^2_r)dr+\int_{s}^{t}\sigma(Y^2_r)dW^2_r.$$
	Hence, $Y_t=(\varphi(t,W,Z_t^1),\varphi(t,W,Z_t^2))$ is a solution of equations \eqref{Sec1:equi-eq} on $[a,b]$ for almost every given $W\in\Omega$. For almost every given $W\in\Omega$. According to Theorem \ref{Sec3:Thm3.1}, Lemma \ref{Sec3:Lemma 3.2} and $Z\in C([a,b];\mathbb{R}^m\times\mathbb{R}^m)$, then $Y\in C([a,b];\mathbb{R}^m\times\mathbb{R}^m)$.  It means that $f(Y^1)$ and $g(Y^2)$ are globally Lipschitz with respect to $Y^1$ and $Y^2$, respectively. Therefore, from the inequalities \eqref{Sec3:3.55}-\eqref{Sec3:3.57},  $Y^1\in C^\alpha([a,b];\mathbb{R}^m)$ and $Y^2\in C^\alpha([a,b];\mathbb{R}^m)$  hold almost every given $W\in\Omega$.  This fact ensures that $Y=(Y^1,Y^2)\in C^{\alpha}([a,b];\mathbb{R}^m\times\mathbb{R}^m)$ holds almost every given $W\in\Omega$.  Finally, the uniqueness can be obtained easily. Even though $f$ and $g$ are locally Lipschitz, it is globally Lipschitz for any $Y^1,Y^2\in C^{\alpha}([a,b];\mathbb{R}^m)$. Hence, similar to proof of \cite[Theorem 3.9]{MR4493559},  the uniqueness can be obtained.
\end{proof}
  
  \section*{\bfseries Acknowledgements}
The author would like to thank Professor Luu Hoang Duc for valuable discussions on the techniques of the Doss--Sussmann transformation.
 \bibliographystyle{abbrv}

\end{document}